\documentclass[12pt]{article}

\usepackage{amsmath,amssymb,amsthm,mathtools}
\usepackage{enumitem}
\usepackage{hyperref}

\usepackage{setspace}
\newtheorem{theorem}{Theorem}[section]
\newtheorem{lemma}[theorem]{Lemma}
\newtheorem{proposition}[theorem]{Proposition}
\newtheorem{corollary}[theorem]{Corollary}

\theoremstyle{definition}
\newtheorem{definition}[theorem]{Definition}
\newtheorem{remark}[theorem]{Remark}
\newtheorem{example}[theorem]{Example}

\newcommand{\R}{\mathbb{R}}
\newcommand{\N}{\mathbb{N}}
\newcommand{\E}{\mathbb{E}}

\newcommand{\pulo}{\vspace{0.2cm}}

\newcommand{\Fseven} {\fontsize{7}{11}\selectfont  }
\newcommand{\Feight} {\fontsize{8}{11}\selectfont  }

\newcommand{\Ften} {\fontsize{10}{11}\selectfont  }

\usepackage{fancyhdr}
\title{Distributional Statistical Models:\\ \large
Weak Moments, Cumulants, and a Central Limit Theorem}
\author{
R. Labouriau\\[0.3em]
\small Department of Mathematics, Aarhus University\\
\small \texttt{rodrigo.labouriau@math.au.dk} and
\small \texttt{rodrigo.labouriau@rlstatlab.com}
}
\date{Spring 2026}

\begin{document}
\maketitle
\setstretch{0.8}
\begin{abstract} \Fseven
Many important statistical models fall outside classical moment-based methods due to the non-existence of moments or moment generating functions. We propose a generalised probabilistic framework in which a probability law is represented by a tempered distribution $T \in \mathcal{S}'$, on the same footing as a density, a distribution function, or a characteristic function. By the classical structure theorem every such distribution is a finite sum of derivatives of ordinary functions, so singular laws (point masses, jumps) are admitted without pathology. Information about the law is extracted by evaluating $T$ on test functions regularised by a given positive Schwartz kernel $\varphi \in \mathcal{S}$---the kernel serving as a probe, not as part of the law. Expectations are defined via the action of distributions on these regularised test functions, yielding well-defined \emph{weak moments}, \emph{weak characteristic functions}, and \emph{weak cumulants} of all orders. These extend classical quantities and retain key algebraic properties such as additivity under independence and natural affine transformation rules.  Despite dispensing with classical densities, the framework recovers ordinary event probabilities: the weak expectation determines a kernel-weighted distribution function, and under the uniqueness conditions developed here, the classical distribution function is fully recovered.

The main results are: (i) a systematic algebra of weak cumulants; (ii) a \emph{weak moment problem} where existence of all moments holds unconditionally and uniqueness depends on the kernel, with uniqueness results under Gaussian kernels (via Hermite completeness, including a multivariate extension to~$\mathbb{R}^d$), positive Schwartz kernels with an exponential tail bound and square-integrable densities (via a Carleman-type criterion), and kernels with exponential decay (via Denjoy--Carleman quasi-analyticity); (iii) a \emph{weak central limit theorem} formulated as convergence of weak characteristic functions to a Gaussian limit, covering cases where the classical theorem fails, together with a quantitative refinement that is \emph{uniform over the kernel}---so that the Gaussian limit is independent of the probe used to observe the law; and (iv) a \emph{distributional CLT} showing that, in the density case, the kernel-weighted measures converge in distribution to the standard normal---thereby bridging weak transform convergence and classical convergence in distribution.

The framework is illustrated with Student's $t$, stable, and hyperbolic distributions. As a statistical consequence, the weak first moment yields a consistent estimator of the location parameter in the Cauchy model, where no classical moment-based estimator exists. A full statistical treatment is given in a companion paper ~\cite{LabouriauA2}.
\end{abstract}

\newpage
\Feight \tableofcontents
\newpage
\setstretch{1.0}
\normalsize

\section{Introduction}
\label{sec:introduction}

Classical statistical inference proceeds along two main routes, each resting on an assumption about the law that can fail.  \emph{Likelihood-based} inference assumes that the laws of a model are dominated by a common reference measure, so that densities---and hence a likelihood---exist; \emph{moment-} and \emph{transform-based} inference instead summarises a law through its moments, cumulants, and characteristic function~\cite{Billingsley1995,van2000asymptotic}.  The likelihood route fails, most fundamentally, when the laws share no common dominating measure---as when the support of $P_\theta$ depends on the parameter---so that no joint density, and hence no likelihood, exists~\cite{LabouriauTransversality}.  Even when a density exists, the likelihood may be inferentially defective: in singular models the Fisher information is degenerate and the Fisher--Rao differential geometry built on it breaks down~\cite{Watanabe2009}, while in certain finite mixtures the score function is biased as an estimating equation, rendering the maximum-likelihood estimator inconsistent~\cite{LabouriauMixture}; a density that merely lacks a closed form, as for general $\alpha$-stable laws, is a milder difficulty.  The moment route fails, independently, for heavy-tailed laws---the Cauchy distribution, low-degree Student~$t$ distributions, and $\alpha$-stable laws---whose moments above a finite order do not exist, so that moment and cumulant methods are unavailable and the characteristic function is non-differentiable at the origin~\cite{Feller1971,Samorodnitsky1994}.  The present paper, and its companion on estimation~\cite{LabouriauA2}, work on this moment--transform side; the more serious likelihood degeneracy is taken up, on the same representational basis, in the companion papers~\cite{LabouriauPaperD,LabouriauTransversality}.

We propose a framework in which a probability law is represented by a tempered distribution $T\in\mathcal{S}'$---on the same footing as a density, a cumulative distribution function, or a characteristic function---and is \emph{probed} through its pairing with a given Schwartz kernel $\varphi\in\mathcal{S}$.  Expectations are defined via this pairing, $\E_{T,\varphi}[\psi]=\langle T,\psi\varphi\rangle$; taking $\psi(x)=x^n$ gives the $n$-th \emph{weak moment} $\langle T,x^n\varphi\rangle$, and the closure of $\mathcal{S}$ under multiplication by polynomials immediately shows that weak moments of all orders exist for every pair $(T,\varphi)$.  When the ordinary moments of the underlying distribution exist, they coincide with the corresponding weak moments up to the kernel factor and are recovered from them in the limit $\varphi\to1$ (made precise in Lemma~\ref{lem:phi-to-1}); when they do not (as in the Cauchy case), the weak moments are still well defined.

It is essential to separate the object being modelled from the instrument used to examine it.  In the distributional framework, a statistical model is a family $\{T_\theta:\theta\in\Theta\}\subset\mathcal{S}'$ of laws; the kernel $\varphi$ is \emph{not} part of the law but a probe through which the law is observed---exactly as a law is observed through its moments (probing with the monomials $x^n$), through its characteristic function (probing with the exponentials $e^{itx}$), or through its distribution function (probing with indicators of half-lines).  The weak moments $\langle T_\theta,x^n\varphi\rangle$ and the weak characteristic function $\langle T_\theta,e^{itx}\varphi\rangle$ are accordingly \emph{instrument readings}: they are properties of the pair $(T,\varphi)$ by design, just as ``the third moment'' is a property of a law together with the decision to look at it through $x^3$.  This places the construction at the end of a long lineage---coefficients (de~Moivre), generating functions (Euler, Laplace), moments (Chebyshev, Markov, Stieltjes), characteristic functions (L\'evy, Cram\'er)---in which a law is known through what it does to a family of test objects, the probe $x^n\varphi$ being the most flexible member to date~\cite{LabouriauDeMoivre}.  Letting $T$ be a general tempered distribution is, moreover, no step into pathology: by the structure theorem these are finitely-differentiated ordinary functions---point masses, jumps, and the like---so singular laws sit on a continuum with smooth ones rather than apart from them (Remark~\ref{rem:structure}).

This separation also answers a natural objection.  In the density case the kernel-weighted object $\varphi f$ is itself a finite measure and the weak moments are its ordinary moments, so one might ask whether the framework is merely the moment theory of a mollified density.  It is not, for three reasons.  First, the probe leaves the law untouched: it multiplies $T$ by $\varphi$ rather than convolving (Remark~\ref{rem:mollification}), so $T$---not a smoothed surrogate---remains the object of study.  Second, and decisively, the probe is \emph{faithful}: the uniqueness theory of Section~\ref{sec:moment_problem} shows that for a positive, sufficiently decaying kernel the family of readings $\{\langle T,x^n\varphi\rangle\}_{n\ge0}$ determines $T$ uniquely, so no information about the law is lost in viewing it through $\varphi$.  Third, the construction is defined verbatim when $T$ has \emph{no density at all}, where ``mollifying a density'' has no meaning---and it is there, rather than in the density case, that the novelty resides.

\paragraph{A motivating example.}
The symmetric $\alpha$-stable laws ($0<\alpha\le2$) are the prototype the framework is built for.  They are specified \emph{operationally}, through the characteristic function $e^{-c|t|^\alpha}$, and---apart from the Gaussian ($\alpha=2$), the Cauchy ($\alpha=1$), and the L\'evy ($\alpha=\tfrac12$) cases---admit no density in closed form; moreover their classical absolute moments of order $\ge\alpha$ diverge, so there is no variance for any $\alpha<2$ and no mean once $\alpha\le1$.  Here there is no density to mollify and no convergent moment sequence to invert; yet the law is a bona fide tempered distribution $T$, and every weak moment $\langle T,x^n\varphi\rangle=\int x^n\varphi\,d\mu$ is finite because $x^n\varphi\in\mathcal{S}(\R)$.  Through a Gaussian kernel the weak characteristic function is the convolution of the stable characteristic function with a Gaussian---a generalised Voigt profile---which for $\alpha=1$ reduces to the Faddeeva/$\operatorname{erfc}$ form computed in the companion paper~\cite{LabouriauA2}, and the weak moments are read from its derivatives at the origin.  The Cauchy case ($\alpha=1$), worked in closed form in Section~\ref{subsec:examples_t}, is the computational specialisation; the general stable case is the one that cannot be reduced to the moments of a mollified density.  A complementary example sits on the likelihood side rather than the moment side---the non-dominated family $\tfrac12\delta_\theta+\tfrac12 N(0,1)$, for which no likelihood exists yet the weak representation applies unchanged (Remark~\ref{rem:moving-atom}).

Giving up the density might seem to come at a price: does one not also give up probabilistic content---event probabilities, quantiles, tail probabilities?  The crucial point is that \emph{it does not}.  In Section~\ref{subsec:preliminaries_spaces} we show that the weak expectation determines a kernel-weighted distribution function by approximating indicator functions with Schwartz functions (Proposition~\ref{prop:weighted-cdf-recovery}), and that the ordinary distribution function is recovered whenever the underlying density is identifiable from the weak representation (Corollary~\ref{cor:cdf-recovery}).  The weak framework thus supports the same kinds of probabilistic statements---event probabilities, quantiles, confidence intervals---and even allows statistical inference, just as classical density-based methods do, while extending seamlessly to models where no classical moments exist or densities are not available.  Moreover, in the classical density setting, the framework admits an explicit reconstruction interpretation: weak data determine a regularised version of the density, from which the original density can be recovered by stable inverse methods (see Appendix~\ref{app:reconstruction}).

\pulo

 \noindent
The paper develops five blocks of results:
\begin{enumerate}[label=(\roman*)]
\item \emph{Algebra of weak cumulants} (Sections~\ref{sec:independence}--\ref{sec:cumulants}): we show that weak characteristic functions factorise under independence, that weak cumulants are additive, and that affine transformation rules carry over from the classical theory.
\item \emph{The weak moment problem} (Section~\ref{sec:moment_problem}): existence of weak moments of all orders is automatic; uniqueness depends on the kernel.  We establish three levels of uniqueness: (a)~for any positive Schwartz kernel with an exponential tail bound, among square-integrable densities via Carleman's condition; (b)~for kernels with exponential-type decay, in all of $\mathcal{S}'(\R)$ via the Denjoy--Carleman theorem; (c)~for Gaussian kernels, via the completeness of Hermite functions in $\mathcal{S}(\R)$.  The Gaussian uniqueness theorem is extended to~$\R^d$ using the multivariate Hermite basis.
\item \emph{Weak central limit theorem} (Section~\ref{sec:clt}): we prove convergence of the weak characteristic function of normalised sums to a Gaussian limit, together with a quantitative (Berry--Esseen type) refinement that is moreover \emph{uniform over the kernel}, so that the Gaussian limit is independent of the probe (instrument-invariance).
\item \emph{Distributional CLT} (Section~\ref{subsec:clt_distributional}): in the density case, the weak CLT is shown to imply classical convergence in distribution of the kernel-weighted measures to the standard normal.  This bridges the weak framework and classical probability, and connects to the Tikhonov reconstruction of Appendix~\ref{app:reconstruction}.
\item \emph{A statistical consequence} (Section~\ref{sec:motivating-consequence}): a short illustration showing that even a single weak moment yields a consistent estimator for the location parameter of a Cauchy distribution, where no classical moment-based estimator exists.
\end{enumerate}

Classical moment theory provides a rich body of determinacy and
indeterminacy results---a measure possessing moments of all orders may
or may not be uniquely determined by them.  The theory of
moment-determinacy (M-determinacy) includes Carleman's sufficient
condition for uniqueness, Krein's criterion for
indeterminacy~\cite{Stoyanov2000Krein}\footnote{The criterion is due to
M.~G.~Krein; as the original work is difficult to obtain, we cite throughout
the accessible treatment of Stoyanov~\cite{Stoyanov2000Krein}.}, and the existence of continuous and
discrete Stieltjes classes of distributions sharing the same ordinary
moments.  Canonical examples include the lognormal distribution and
related heavy-tailed laws; see Berg~\cite{Berg1988},
Stoyanov~\cite{Stoyanov2013}, and the geometric
perspective in~\cite{StoyanovInverardiTagliani2023}.  The present work
does not replace the classical moment problem, but rather introduces a
kernel-weighted moment calculus in which the separating family is
changed from the bare monomials $x^n$ to the weighted functions
$x^n\varphi(x)$.  The relation between weak uniqueness and classical
M-determinacy is discussed in
Section~\ref{subsec:classical-moment-problem}.

Section~\ref{sec:preliminaries} introduces the framework, basic definitions, and the recovery of event probabilities from weak expectations.  Section~\ref{sec:independence} develops independence and additivity.  Section~\ref{sec:cumulants} studies weak transforms and cumulants.  Section~\ref{sec:examples} presents illustrative examples.  Section~\ref{sec:moment_problem} is the centrepiece: the weak moment problem and its uniqueness theorems.  Section~\ref{sec:multivariate} extends the framework to~$\R^d$.  Section~\ref{sec:clt} establishes the weak CLT, its quantitative refinement (including a kernel-uniform version), and the distributional CLT for kernel-weighted measures.  Section~\ref{sec:motivating-consequence} gives the statistical illustration.  Section~\ref{sec:discussion} concludes.  Appendix~\ref{app:leibniz} presents a technical estimate used in the proof of Theorem~\ref{thm:dc-unique}. Additional proofs omitted from the main text for brevity are collected in Appendix~\ref{app:proofs}.  Appendix~\ref{app:reconstruction} develops Tikhonov reconstruction from weak data.

\paragraph{Notation.}
Throughout the paper, weak objects carry a superscript $\varphi$ denoting the kernel in play: ${}^{(\varphi)}m_n$ for weak moments, ${}^{(\varphi)}\phi(t)$ for the weak characteristic function, ${}^{(\varphi)}K(t)$ for the weak cumulant generating function, and ${}^{(\varphi)}\kappa_n$ for weak cumulants.  Classical objects (when they exist) carry no superscript.  

\section{Preliminaries and generalised probability framework}
\label{sec:preliminaries}

\subsection{Distribution--kernel pairs and weak expectation}
\label{subsec:preliminaries_spaces}

Let $\mathcal{S}(\R^d)$ denote the Schwartz space of rapidly decreasing smooth functions, and $\mathcal{S}'(\R^d)$ its dual, the space of tempered distributions.  We work primarily in the tempered setting, since it is stable under Fourier transform and accommodates the weak characteristic functions introduced below.  The closure of $\mathcal{S}(\R^d)$ under multiplication by polynomials---if $\varphi\in\mathcal{S}(\R^d)$ and $p$ is a polynomial, then $p\varphi\in\mathcal{S}(\R^d)$---underlies the existence of weak moments of all orders.

\begin{definition}[Distribution--kernel pair]
A \emph{distribution--kernel pair} is a tuple $(T,\varphi)$ with $T \in \mathcal{S}'(\R^d)$ and $\varphi \in \mathcal{S}(\R^d)$.  The tempered distribution $T$ \emph{represents the law}---the object of inference, in the role a density would otherwise play---while the kernel $\varphi$ is the \emph{probe} through which the law is examined: it regularises the pairing but is not itself part of the model.
\end{definition}

\begin{remark}[Singularity as differentiated regularity]\label{rem:structure}
Admitting a general $T\in\mathcal{S}'(\R^d)$ does not admit arbitrarily
pathological objects.  By the classical structure theorem, every tempered
distribution is a finite sum of derivatives of continuous functions of at most
polynomial growth, $T=\sum_{|\alpha|\le m}D^\alpha f_\alpha$
(Strichartz~\cite{Strichartz2003}, \S6.3; H\"ormander~\cite{Hormander1990}).
Singular probabilistic features---point masses, jumps, cusps---are accordingly
\emph{differentiated regularity}: they arise by differentiating ordinary
functions in the weak sense (a point mass is the derivative of a step), and the
kernel $\varphi$ converts them back into stable scalar quantities.  Heavy tails
are a separate matter: there $T$ is an ordinary function of slow decay, and it
is the rapid decay of $\varphi$---not the structure theorem---that makes
$x^n\varphi\in\mathcal{S}(\R^d)$ and renders the weak moments finite.  The
single pairing thus accommodates both departures from classical
regularity---singular $T$ and heavy tails---for different reasons.
\end{remark}

\begin{remark}\label{rem:univariate-focus}
The definitions are stated for $\R^d$ to establish generality, but the detailed development is carried out for $d=1$.  The multivariate extension is straightforward for the basic constructions; a systematic treatment is left for future work.
\end{remark}

Given $(T,\varphi)$, define the \emph{class of admissible functions} by $\mathcal{A}_\varphi := \{ \psi : \psi\varphi \in \mathcal{S}(\R^d) \}$.  All polynomials and all functions $x\mapsto e^{itx}$ belong to $\mathcal{A}_\varphi$.

\begin{definition}[Generalised expectation]
For $(T,\varphi)$ and $\psi\in\mathcal{A}_\varphi$, the \emph{weak expectation} is
\[
\E_{T,\varphi}[\psi] \;:=\; \langle T,\,\psi\varphi\rangle.
\]
\end{definition}

\begin{definition}[Generalised probability measure]
A pair $(T,\varphi)$ defines a \emph{generalised probability measure} when the probe $\varphi$ exhibits $T$ as a law: $\langle T,\varphi\rangle = 1$ (normalisation) and $\langle T,\psi\varphi\rangle \geq 0$ for all non-negative $\psi\in\mathcal{A}_\varphi$ (positivity).
\end{definition}

\begin{proposition}[Positivity implies a genuine measure]\label{prop:positivity-measure}
Let $(T,\varphi)$ be a generalised probability measure with
$\varphi(x)>0$ for all~$x$.  Then the distributional product
$T\varphi\in\mathcal{S}'(\R)$ is a positive Radon measure: there
exists a unique positive Borel measure~$\mu$ on~$\R$ with
$\langle T,\psi\varphi\rangle = \int\psi\,d\mu$ for all
$\psi\in\mathcal{A}_\varphi$.
\end{proposition}

\begin{proof}
Since $\varphi>0$ everywhere, every non-negative $\psi\in C_c^\infty(\R)$
satisfies $\psi\varphi\in C_c^\infty(\R)\subset\mathcal{S}(\R)$, so
$\psi\in\mathcal{A}_\varphi$.  The positivity condition therefore gives
$\langle T,\psi\varphi\rangle\ge 0$ for all non-negative
$\psi\in C_c^\infty(\R)$.  In other words, the distribution
$T\varphi\in\mathcal{D}'(\R)$ is positive.  By the Schwartz theorem on
positive distributions~\cite[Thm.~2.1.7]{Hormander1990}, $T\varphi$ is a
positive Radon measure~$\mu$.  Uniqueness of~$\mu$ follows from the
density of $C_c^\infty(\R)$ in $C_0(\R)$.
\end{proof}

\begin{remark}\label{rem:positivity-strong}
For pairs $(T_f,\varphi)$ with $f\ge 0$ a.e.\ and $\varphi>0$, the
positivity condition is automatic and $\mu = f\varphi\,dx$.  The
technical results of this paper---weak moments, characteristic
functions, additivity, the CLT, and the uniqueness theorems---depend
only on the normalisation condition and the structure of the pair,
not on positivity.  Positivity is thus a modelling axiom that
restricts attention to pairs with a probabilistic interpretation.
Beyond $\varphi>0$, the only structural hypothesis is $T\in\mathcal{S}'(\R)$:
no smoothness, decay, or density of $T$ is assumed.
\end{remark}

This point of view leads naturally to a generalised notion of random variable.

\begin{definition}[Generalised random variable]
A generalised random variable is an abstract object \(X\) whose law is the
tempered distribution \(T \in \mathcal{S}'(\mathbb{R})\), examined through a
kernel \(\varphi \in \mathcal{S}(\mathbb{R})\), and whose expectations are the
corresponding probe-readings
\[
E[X;\psi] := E_{T,\varphi}[\psi] = \langle T, \psi \varphi \rangle,
\]
for all admissible test functions \(\psi\).
\end{definition}

\begin{remark}
In this formulation, a generalised random variable is not defined as a
measurable function on an underlying probability space, but rather through
its expectation functional---a purely functional-analytic object. In this
respect it is a scalar, regularised instance of the generalised random
variables of Gel'fand and Vilenkin~\cite{Gelfand1964}, which are likewise
specified by their action on test functions rather than as pointwise maps on
a sample space. This extends the classical notion of random
variable while retaining its probabilistic content. In the density case,
\(T = T_f\), this reduces to
\[
E[X;\psi] = \int \psi(x)\varphi(x) f(x)\,dx,
\]
which coincides with classical expectation in the limit \(\varphi \to 1\).
\end{remark}

For a classical density $f\in L^1(\R^d)$ with associated distribution $T_f$, the weak expectation reduces to
\[
\E_{T_f,\varphi}[\psi] = \int_{\R^d} f(x)\psi(x)\varphi(x)\,dx,
\]
which recovers the ordinary expectation in the limit $\varphi\to 1$, made
precise as follows.

\begin{lemma}[Recovery in the limit $\varphi\to1$]\label{lem:phi-to-1}
Let $f\in L^1(\R)$ and let $\psi$ be measurable with $\psi f\in L^1(\R)$.
Fix $\rho\in\mathcal{S}(\R)$ with $0\le\rho\le1$ and $\rho(0)=1$ (for
instance $\rho(x)=e^{-x^2/2}$), and set $\varphi_\varepsilon(x):=\rho(\varepsilon x)$
for $\varepsilon>0$.  Then $\varphi_\varepsilon\in\mathcal{S}(\R)$,
$\varphi_\varepsilon>0$, $\varphi_\varepsilon\to1$ pointwise as
$\varepsilon\downarrow0$, and
\[
\E_{T_f,\varphi_\varepsilon}[\psi]
=\int_\R \psi(x)\varphi_\varepsilon(x)f(x)\,dx
\;\xrightarrow[\varepsilon\downarrow0]{}\;
\int_\R\psi(x)f(x)\,dx .
\]
In particular, $\psi(x)=x^n$ recovers the ordinary $n$-th moment whenever it
exists, and $\psi(x)=e^{itx}$ recovers the classical characteristic function
for every $t\in\R$.
\end{lemma}

\begin{proof}
For each $\varepsilon>0$, $\rho(\varepsilon\,\cdot)\in\mathcal{S}(\R)$ and
$|\psi\,\varphi_\varepsilon f|\le\|\rho\|_\infty\,|\psi f|\le|\psi f|\in
L^1(\R)$; since $\varphi_\varepsilon(x)\to\rho(0)=1$ for every $x$, dominated
convergence gives the stated limit.  For $\psi(x)=e^{itx}$ the dominating
function $|f|$ is integrable for all $t$, so the recovery is unconditional.
\end{proof}

\begin{remark}\label{rem:mollification}
The construction $\E_{T,\varphi}[\psi] = \langle T,\psi\varphi\rangle$ is related to, but distinct from, classical mollification $T*\rho_\varepsilon$.  Mollification acts by \emph{convolution} (global averaging), whereas the present framework acts by \emph{pointwise multiplication} by the kernel~$\varphi$, localising information to the region where $\varphi$ is appreciably nonzero.  This distinction becomes important in the reconstruction problem, where the forward operator is a multiplication rather than a convolution.
\end{remark}

\pulo

A natural question arises at this point: if expectations are defined via the distributional pairing $\langle T,\psi\varphi\rangle$ rather than via integration against a density, can one still recover ordinary event probabilities?  The answer is affirmative: the weak expectation determines a kernel-weighted distribution function, and under the uniqueness conditions of Section~\ref{sec:moment_problem}, the classical distribution function is fully recovered.

\begin{proposition}[Recovery of weighted distribution functions]
\label{prop:weighted-cdf-recovery}
Let $f\in L^1(\R)$ be a nonnegative density, let $T_f$ be the induced
tempered distribution, and let $\varphi\in\mathcal{S}(\R)$ with
$\varphi\ge 0$. For each $a\in\R$, define
\[
F_{\varphi}(a):=\int_{-\infty}^{a} f(x)\varphi(x)\,dx .
\]
Then $F_{\varphi}(a)$ is determined by the weak expectation
$\E_{T_f,\varphi}$ in the following sense: if $(\psi_n)_{n\ge1}\subset
\mathcal{S}(\R)$ satisfies
\[
0\le \psi_n\le 1,\qquad \psi_n(x)\to \mathbf{1}_{(-\infty,a)}(x)
\quad\text{pointwise},
\]
then
\[
\E_{T_f,\varphi}[\psi_n]\longrightarrow F_{\varphi}(a).
\]
In particular, the weak expectation determines the kernel-weighted cumulative
distribution function $F_{\varphi}$.
\end{proposition}

\begin{proof}
Since $T_f$ is induced by the density $f$, the weak expectation has the
integral form
\[
\E_{T_f,\varphi}[\psi_n]
=\int_{\R}\psi_n(x)\varphi(x)f(x)\,dx.
\]
By assumption, $0\le \psi_n\le 1$ and $\psi_n(x)\to \mathbf{1}_{(-\infty,a)}(x)$
pointwise. Hence
\[
0\le \psi_n(x)\varphi(x)f(x)\le \varphi(x)f(x),
\]
and the dominating function $\varphi f$ is integrable because
$\varphi\in\mathcal{S}(\R)$ is bounded and $f\in L^1(\R)$.
Therefore, by the dominated convergence theorem,
\[
\lim_{n\to\infty}\E_{T_f,\varphi}[\psi_n]
=
\int_{\R}\mathbf{1}_{(-\infty,a)}(x)\varphi(x)f(x)\,dx
=
\int_{-\infty}^{a} f(x)\varphi(x)\,dx = F_\varphi(a).
\]
\end{proof}

\begin{corollary}[Recovery of the classical distribution function]
\label{cor:cdf-recovery}
Assume that the kernel $\varphi$ is fixed and that the weak data
determine $f$ uniquely---for example, under the uniqueness conditions developed
in Section~\ref{sec:moment_problem}. Then the ordinary cumulative distribution function
\[
F(a)=\int_{-\infty}^{a} f(x)\,dx
\]
is recovered from the weak representation.
\end{corollary}

\begin{proof}
By Proposition~\ref{prop:weighted-cdf-recovery}, the weak expectation
determines $F_{\varphi}(a)=\int_{-\infty}^{a}f(x)\varphi(x)\,dx$ for every
$a$. The non-trivial step is the uniqueness of $f$ from $F_\varphi$: under the
assumptions of Section~\ref{sec:moment_problem}, the weak representation
determines $f$ uniquely. Once $f$ is known, the ordinary cumulative
distribution function $F(a)=\int_{-\infty}^{a}f(x)\,dx$ is obtained by direct integration.
\end{proof}

\begin{remark}\label{rem:probability-recovery}
Proposition~\ref{prop:weighted-cdf-recovery} shows that the weak expectation
determines event probabilities only indirectly, through approximation by
admissible smooth functions. What is obtained directly is the
kernel-weighted distribution function $F_{\varphi}$. The ordinary
distribution function is then recovered whenever the underlying density is
identifiable from the weak representation.  Thus, despite the absence of a classical density in the general setting, the weak framework does not sacrifice probabilistic content: it supports the same kinds of probabilistic statements---event probabilities, quantiles, tail probabilities---as classical density-based inference.
\end{remark}

\begin{remark}[Reconstruction from weak data]
\label{rem:reconstruction}
In the classical density setting, the weak representation determines the
product $g(x)=\varphi(x)f(x)$. Recovering $f$ from $g$ is an inverse problem,
which is in general ill-posed due to the decay of $\varphi$ at infinity.
However, stable reconstruction can be achieved by standard regularisation
methods. In Appendix~\ref{app:reconstruction}, we show that Tikhonov
regularisation yields a well-posed recovery procedure in $L^2(\R)$, with
explicit error bounds under noise.
\end{remark}

\subsection{Weak moments, transforms, and cumulants}
\label{subsec:preliminaries_moments}

\begin{definition}[Weak moments]\label{def:weak-moments}
For $n\geq 0$, the $n$-th \emph{weak moment} is
\[
{}^{(\varphi)}m_n \;:=\; \E_{T,\varphi}[x^n] \;=\; \langle T,\,x^n\varphi(x)\rangle.
\]
\end{definition}

\begin{proposition}\label{prop:moments-exist}
All weak moments are well defined: ${}^{(\varphi)}m_n$ exists for every $n\geq 0$ and every pair $(T,\varphi)$.
\end{proposition}

\begin{proof}
Since $\varphi \in \mathcal{S}(\R^d)$ and $\mathcal{S}$ is closed under multiplication by polynomials, the product $x^n\varphi(x)$ belongs to $\mathcal{S}(\R^d)$.  As $T \in \mathcal{S}'(\R^d)$, the pairing $\langle T, x^n\varphi(x)\rangle$ is well defined for every $n$.
\end{proof}

\begin{remark}\label{rem:weak-vs-classical}
For a classical density $f$ and pair $(T_f,\varphi)$, the weak moment is ${}^{(\varphi)}m_n = \int x^n f(x)\varphi(x)\,dx$, which differs from the classical moment $m_n = \int x^n f(x)\,dx$ unless $\varphi \equiv 1$.  In the limit $\varphi \to 1$, the two coincide whenever $m_n$ exists.  For a general kernel, ${}^{(\varphi)}m_n$ is a \emph{regularised} version of $m_n$.
\end{remark}

\begin{definition}[Weak characteristic function]\label{def:wcf}
The \emph{weak characteristic function} of a generalised probability measure $(T,\varphi)$ on $\R$ is
\[
{}^{(\varphi)}\phi(t)
\;:=\;
\E_{T,\varphi}[e^{itx}]
\;=\;
\langle T,\,e^{itx}\varphi(x)\rangle,
\qquad t\in\R.
\]
\end{definition}

\begin{definition}[Weak cumulant generating function and weak cumulants]\label{def:cumulants}
Assume ${}^{(\varphi)}\phi(t)\neq 0$ in a neighbourhood of $t=0$.  The \emph{weak cumulant generating function} is ${}^{(\varphi)}K(t):=\log {}^{(\varphi)}\phi(t)$, and the $n$-th \emph{weak cumulant} is
\[
{}^{(\varphi)}\kappa_n
\;:=\;
\left.\frac{d^n}{dt^n}{}^{(\varphi)}K(t)\right|_{t=0}.
\]
\end{definition}

\begin{remark}
Weak moments, weak characteristic functions, and weak cumulants all depend on the kernel~$\varphi$, as the superscript notation makes explicit.  They are properties of the pair $(T,\varphi)$, not of $T$ alone.  The kernel should therefore be viewed as part of the model specification.
\end{remark}

\begin{remark}[Stability of weak measurements]\label{rem:continuity}
Tempered distributions act \emph{continuously} on Schwartz space: if
$\varphi_j\to\varphi$ in $\mathcal{S}(\R^d)$ and $T\in\mathcal{S}'(\R^d)$, then
$\langle T,\varphi_j\rangle\to\langle T,\varphi\rangle$.  The dependence of the
weak moments, cumulants, and characteristic function on the kernel is therefore
not merely formal but continuous in the natural topology of the test space:
small smooth changes of the probe produce small changes of the reading.  This
is the analytic basis of the kernel-uniform behaviour exploited in
Section~\ref{sec:clt} (Theorem~\ref{thm:weak-BE-uniform}), and it makes the weak
framework an intrinsically regularised one rather than a merely formal
extension of moment theory.
\end{remark}

\subsection{The running Student's \texorpdfstring{$t$}{t} example}
\label{subsec:preliminaries_kernel_t}

As a guiding example throughout the paper, we use the Student's $t_\nu$ family with density $f_\nu(x)=c_\nu(1+x^2/\nu)^{-(\nu+1)/2}$.  Classically, $\E[|X|^k]<\infty$ iff $k<\nu$; for $\nu=1$ (Cauchy), no non-trivial ordinary moment exists.

For any $\varphi\in\mathcal{S}(\R)$, the pair $(T_\nu,\varphi)$ defines a generalised probability measure, and the weak moments ${}^{(\varphi)}m_k = \E_{T_\nu,\varphi}[x^k]$ are well defined for all $k\geq 0$, regardless of~$\nu$.  We shall illustrate the main results with this family as we proceed.

\section{Independence and additivity of weak cumulants}
\label{sec:independence}

\subsection{Product representations and independence}
\label{subsec:independence_product}

Let $(T_1,\varphi_1)$ and $(T_2,\varphi_2)$ be generalised probability measures on $\R$.  Their \emph{product} is the pair $(T_1\otimes T_2,\,\varphi_1\otimes\varphi_2)$ on $\R^2$, where $\otimes$ denotes the tensor product.

\begin{definition}[Independence]
Two generalised random variables represented by $(T_1,\varphi_1)$ and $(T_2,\varphi_2)$ are \emph{independent} if their joint representation is the product pair $(T_1\otimes T_2,\,\varphi_1\otimes\varphi_2)$.
\end{definition}

\begin{remark}
Independence depends on the chosen representation, since the kernel is part of the model specification.  This is analogous to the role of a $\sigma$-algebra in conditional independence.  In practice, the kernel is fixed at the modelling stage, and independence is well defined relative to that choice.
\end{remark}

\begin{remark}[Classical consistency]\label{rem:independence-classical}
When $T_i = T_{f_i}$ for densities $f_i\ge 0$, the tensor product
$T_{f_1}\otimes T_{f_2}$ is the distribution induced by the product
density $f_1(x)f_2(y)$:
$\langle T_{f_1}\otimes T_{f_2},\Psi\rangle
= \int\!\int f_1(x)f_2(y)\Psi(x,y)\,dx\,dy$.
Thus, in the classical setting, the definition above reduces to the
standard product-density characterisation of independence.
\end{remark}

\subsection{Weak characteristic functions of sums}
\label{subsec:independence_cf}

Let $X$ and $Y$ be independent generalised random variables with pairs $(T_1,\varphi_1)$ and $(T_2,\varphi_2)$, and set $Z=X+Y$.

\begin{proposition}\label{prop:cf-factorise}
The weak characteristic function factorises:
\[
{}^{(\varphi)}\phi_Z(t) = {}^{(\varphi_1)}\phi_1(t)\,{}^{(\varphi_2)}\phi_2(t).
\]
\end{proposition}

\begin{proof}
By the defining property of the tensor product,
\begin{align*}
{}^{(\varphi)}\phi_Z(t)
&= \langle T_1 \otimes T_2,\,
e^{it(x+y)}\varphi_1(x)\varphi_2(y)\rangle \\
&= \langle T_1,\,e^{itx}\varphi_1(x)\rangle\;
\langle T_2,\,e^{ity}\varphi_2(y)\rangle
= {}^{(\varphi_1)}\phi_1(t)\;{}^{(\varphi_2)}\phi_2(t).
\end{align*}
\end{proof}

\subsection{Additivity of weak cumulants}
\label{subsec:independence_cumulants}

\begin{theorem}\label{thm:cumulant-additivity}
Let $X$ and $Y$ be independent generalised random variables.  Then ${}^{(\varphi)}K_{X+Y}(t) = {}^{(\varphi_1)}K_X(t) + {}^{(\varphi_2)}K_Y(t)$, whenever the logarithms are defined.
\end{theorem}

\begin{proof}
From Proposition~\ref{prop:cf-factorise}, $\phi_{X+Y}(t)=\phi_X(t)\phi_Y(t)$; taking logarithms yields the result.
\end{proof}

\begin{corollary}\label{cor:cumulant-additivity}
If ${}^{(\varphi_1)}\phi_X(0)\neq 0$ and ${}^{(\varphi_2)}\phi_Y(0)\neq 0$, then the weak cumulants are additive: ${}^{(\varphi)}\kappa_n(X+Y) = {}^{(\varphi_1)}\kappa_n(X) + {}^{(\varphi_2)}\kappa_n(Y)$ for all $n \geq 1$.
\end{corollary}

\noindent\textit{Proof.}\enspace See Appendix~\ref{app:proofs}.\medskip

\subsection{Affine transformations}
\label{subsec:independence_transformations}

\begin{proposition}[Translation]\label{prop:translation}
Let $a \in \R$ and define the translated pair $(T_a,\varphi_a)$ by $\langle T_a,\psi\rangle := \langle T,\psi(\cdot + a)\rangle$ and $\varphi_a(x):=\varphi(x-a)$.  Then ${}^{(\varphi_a)}\phi_a(t)=e^{ita}{}^{(\varphi)}\phi(t)$, and
\[
{}^{(\varphi_a)}\kappa_1 = {}^{(\varphi)}\kappa_1 + a,
\qquad
{}^{(\varphi_a)}\kappa_n = {}^{(\varphi)}\kappa_n \quad (n \geq 2).
\]
\end{proposition}

\noindent\textit{Proof.}\enspace See Appendix~\ref{app:proofs}.\medskip

\begin{remark}[Kernel-shift convention]\label{rem:kernel-shift}
In Proposition~\ref{prop:translation} both the distribution and the
kernel shift together: $T\mapsto T_a$, $\varphi\mapsto\varphi_a$.
This yields the clean phase rule
${}^{(\varphi_a)}\phi_a(t)=e^{ita}\,{}^{(\varphi)}\phi(t)$ and
preserves the algebraic structure of cumulants.  An alternative
convention---keeping the kernel fixed and shifting only the
distribution---corresponds to measuring the translated random
variable $X+a$ with the \emph{same} instrument~$\varphi$.  In that
case the weak expectation becomes
$\E_{T_a,\varphi}[\psi] = \langle T,
\psi(\cdot+a)\varphi(\cdot+a)\rangle$, which mixes the shift of~$T$
with the localisation region of~$\varphi$.  The ``fixed kernel''
perspective is the natural setting for the observation-operator
framework developed in the companion paper~\cite{LabouriauPaperD},
where the kernel plays the role of a fixed measurement device.  Both
conventions are valid; the present paper adopts the ``co-moving
kernel'' convention throughout.
\end{remark}

\begin{proposition}[Scaling]\label{prop:scaling}
Let $b \neq 0$ and define $(T_b,\varphi_b)$ by $\E_{T_b,\varphi_b}[\psi] := \E_{T,\varphi}[\psi(b\,\cdot)]$ and $\varphi_b(x):=\varphi(bx)$.  Then ${}^{(\varphi)}\phi_b(t)={}^{(\varphi)}\phi(bt)$ and ${}^{(\varphi_b)}\kappa_n = b^n{}^{(\varphi)}\kappa_n$.
\end{proposition}

\noindent\textit{Proof.}\enspace See Appendix~\ref{app:proofs}.\medskip

\section{Weak transforms and weak cumulants}
\label{sec:cumulants}

\subsection{Smoothness of weak characteristic functions}
\label{subsec:cumulants_cf}

\begin{proposition}\label{prop:wcf-smooth}
The weak characteristic function ${}^{(\varphi)}\phi(t)$ is infinitely differentiable in $t$, and for each $n \geq 0$,
\[
({}^{(\varphi)}\phi)^{(n)}(t)
=
\E_{T,\varphi}[(ix)^n e^{itx}]
=
\langle T, (ix)^n e^{itx}\varphi(x) \rangle.
\]
\end{proposition}

\noindent\textit{Proof.}\enspace See Appendix~\ref{app:proofs}.\medskip

\begin{corollary}\label{cor:moments-from-cf}
The derivatives at the origin encode the weak moments:
$({}^{(\varphi)}\phi)^{(n)}(0) = i^n\, {}^{(\varphi)}m_n$.
\end{corollary}

\noindent\textit{Proof.}\enspace See Appendix~\ref{app:proofs}.\medskip

\begin{remark}
Thus ${}^{(\varphi)}\phi$ is smooth and encodes all weak moments, even when the classical characteristic function is not differentiable at the origin (as for stable laws with $\alpha<2$).
\end{remark}

\subsection{Existence of weak cumulants and formal expansion}
\label{subsec:cumulants_cgf}

\begin{proposition}\label{prop:cumulants-exist}
If all weak moments exist and ${}^{(\varphi)}\phi(0)\neq 0$, then all weak cumulants ${}^{(\varphi)}\kappa_n$ exist.
\end{proposition}

\begin{proof}
Since ${}^{(\varphi)}\phi$ is smooth and non-vanishing near $0$, the function $\log {}^{(\varphi)}\phi$ is smooth.  Its derivatives at the origin are finite combinations of weak moments.
\end{proof}

Whenever ${}^{(\varphi)}K(t)$ is sufficiently regular at $t=0$, it admits the formal expansion
\[
{}^{(\varphi)}K(t)
=
\sum_{n=1}^\infty {}^{(\varphi)}\kappa_n \frac{(it)^n}{n!},
\qquad
{}^{(\varphi)}\phi(t)
=
\exp\left\{
\sum_{n=1}^\infty {}^{(\varphi)}\kappa_n \frac{(it)^n}{n!}
\right\}.
\]
This should be interpreted as a formal power series unless analyticity of ${}^{(\varphi)}K$ near the origin is established.  The second-order truncation,
\[
\log {}^{(\varphi)}\phi(u)
=
i{}^{(\varphi)}\kappa_1 u - \frac{{}^{(\varphi)}\kappa_2}{2}u^2 + o(u^2),
\qquad u \to 0,
\]
follows from smoothness and ${}^{(\varphi)}\phi(0)=1$, and is the key ingredient in the proof of the weak CLT (Section~\ref{sec:clt}).

\subsection{Relation to classical transforms}
\label{subsec:cumulants_classical}

For a classical density $f \in L^1(\R)$ and kernel $\varphi$ approximating the constant function~$1$,
\[
{}^{(\varphi)}\phi(t)
=
\int_{\R} e^{itx} f(x)\varphi(x)\,dx
\;\approx\;
\int_{\R} e^{itx} f(x)\,dx,
\]
recovering the classical characteristic function~\cite{Lukacs1970}.  Thus ${}^{(\varphi)}\phi$ is a regularised version of the classical one, and weak cumulants converge to classical cumulants (when they exist) as $\varphi\to 1$.

\section{Examples}
\label{sec:examples}

\subsection{Student's \texorpdfstring{$t$}{t} distributions: classical failure and weak recovery}
\label{subsec:examples_t}

The Student $t_\nu$ family provides the canonical illustration: classically, no moment of order $\ge\nu$ exists.  In the weak framework, all moments ${}^{(\varphi)}m_k$ exist for every $k$ and every $\nu$.

For the standard Cauchy distribution ($\nu = 1$) with Gaussian kernel $\varphi(x) = e^{-x^2/2}$, the weak moments are
\[
{}^{(\varphi)}m_n = \int_{-\infty}^{\infty} \frac{x^n}{\pi(1+x^2)}\,e^{-x^2/2}\,dx.
\]
All odd weak moments vanish by symmetry.  The zeroth moment admits the
closed form ${}^{(\varphi)}m_0 = e^{1/2}\operatorname{erfc}(1/\sqrt{2})
\approx 0.523$.  To see this, use the partial-fraction identity
$(\pi(1+x^2))^{-1} = (2\pi i)^{-1}(1/(x-i) - 1/(x+i))$ and complete
the square in each integral $\int e^{-x^2/2}/(x\mp i)\,dx$; the result
reduces to the complementary error function via the standard formula
$\int_0^\infty e^{-u^2}\,du = \sqrt{\pi}/2$.  For the second moment,
integration by parts or the identity
$x^2/(1+x^2) = 1 - 1/(1+x^2)$ gives
${}^{(\varphi)}m_2 = \sqrt{2/\pi} - {}^{(\varphi)}m_0 \approx 0.275$.
These values are finite and explicitly computable despite the classical
second moment not existing.

\begin{remark}
The normalisation ${}^{(\varphi)}m_0 \neq 1$ reflects the fact that $\varphi\neq 1$; one may work with the normalised pair $(T_1,\varphi/{}^{(\varphi)}m_0)$ if desired.  The numerical values of the weak moments depend on the kernel bandwidth, illustrating the regularisation-dependence emphasised throughout this paper.
\end{remark}

\subsection{Stable distributions and regularisation of transforms}
\label{subsec:examples_stable}

Let $X$ be a symmetric $\alpha$-stable random variable with $0<\alpha\le 2$, specified \emph{operationally} by its characteristic function $\exp(-c|t|^\alpha)$.  Apart from the Gaussian ($\alpha=2$), Cauchy ($\alpha=1$), and L\'evy ($\alpha=\tfrac12$) cases, the density has no closed form; the classical absolute moments of order $\ge\alpha$ diverge (no variance for any $\alpha<2$, no mean once $\alpha\le1$); and for $0<\alpha<2$ the classical characteristic function fails to be infinitely differentiable at $t=0$.  The law is nonetheless a bona fide tempered distribution $T$, so the framework applies with no density in sight.

In the weak framework, ${}^{(\varphi)}\phi(t) = \langle T, e^{itx}\varphi(x)\rangle$ is smooth in $t$ by construction, and all weak moments and cumulants exist---each $\langle T,x^n\varphi\rangle=\int x^n\varphi\,d\mu$ being finite because $x^n\varphi\in\mathcal{S}(\R)$, irrespective of~$\alpha$.  With a Gaussian kernel $\varphi(x)=e^{-x^2/(2\sigma^2)}$ the weak characteristic function is, up to a constant, the convolution of the stable characteristic function with a Gaussian,
\[
   {}^{(\varphi)}\phi(t)\;\propto\;\bigl(\widehat{\varphi}\ast e^{-c|\cdot|^{\alpha}}\bigr)(t),
\]
a \emph{generalised Voigt profile}, from which derivatives at the origin---and hence all weak cumulants---are read.  For $\alpha=1$ this is the Gaussian--Lorentzian (Voigt) convolution, expressible through the Faddeeva function and reducing to the closed-form $\operatorname{erfc}$ values obtained for the Cauchy law in Section~\ref{subsec:examples_t} and in the companion paper~\cite{LabouriauA2}.  This provides a consistent way to extend moment- and cumulant-based methods to stable distributions, for which neither a closed-form density nor any moment of order $\ge\alpha$ exists classically.

\subsection{Hyperbolic models}
\label{subsec:examples_hyperbolic}

The normal inverse Gaussian distribution $\mathrm{NIG}(\alpha,\beta,\mu,\delta)$ has density involving a modified Bessel function~$K_1$ and tails decaying as $|x|^{-3/2}\exp(-\alpha|x|)$---faster than polynomial but slower than Gaussian.  In the classical setting, all moments exist.  In the weak framework with a Gaussian kernel, the product $f(x)\varphi(x)$ is super-exponentially integrable, weak moments converge rapidly, and the weak characteristic function has super-exponential decay in~$t$, providing improved numerical stability for transform-based inference.  As $\varphi\to 1$, the weak cumulants converge to their classical counterparts.  More generally, the generalised hyperbolic family is accommodated by the same device.

\begin{remark}
The examples above span three regimes: \emph{complete failure} of classical moments (Cauchy, general stable), \emph{partial failure} (Student $t_\nu$ with moderate~$\nu$), and \emph{no failure but improved numerics} (NIG, generalised hyperbolic).  Further examples---elliptically contoured distributions and exponential dispersion models---fit naturally into the framework; we omit the details for brevity, noting only that the multivariate case requires tensor weak moments ${}^{(\varphi)}m_\alpha = \langle T, x^\alpha\varphi(x)\rangle$ indexed by multi-indices.
\end{remark}

\section{The weak moment problem}
\label{sec:moment_problem}

This section contains the main theoretical results of the paper.  We show that the weak moment problem has a fundamentally different character from its classical counterpart: existence is unconditional, and uniqueness is governed entirely by the kernel.  We establish a hierarchy of three uniqueness theorems---for positive Schwartz kernels with an exponential tail bound via Carleman's condition, for kernels with exponential-type decay via Denjoy--Carleman quasi-analyticity, and for Gaussian kernels via completeness of Hermite functions---and show that positivity of the kernel is both sufficient and necessary for uniqueness.

\subsection{Formulation: existence vs uniqueness}
\label{subsec:moment_formulation}

Let $(T,\varphi)$ be a probability pair with weak moments ${}^{(\varphi)}m_n = \langle T, x^n\varphi(x)\rangle$ for $n\ge 0$.  The \emph{weak moment problem} asks: given a sequence $\{m_n\}_{n\ge 0}$, does there exist a pair $(T,\varphi)$ with ${}^{(\varphi)}m_n = m_n$ for all~$n$?  And if so, is the pair uniquely determined?

In contrast with the classical Hamburger moment problem~\cite{Shohat1943}, existence in the weak setting is automatic: for any $T\in\mathcal{S}'(\R)$ and any $\varphi\in\mathcal{S}(\R)$, all weak moments ${}^{(\varphi)}m_n$ exist (Proposition~\ref{prop:moments-exist}).  Thus the weak moment problem shifts the focus entirely to \emph{uniqueness}: can two different distributions $T_1\ne T_2$ produce identical weak moment sequences with respect to the same kernel?

In general, uniqueness fails: both the distributional component and the kernel contribute to the moment sequence, and different pairs may produce the same moments.  If, however, the kernel is \emph{fixed a priori}, the question reduces to identifying $T$ from the sequence $\{\langle T, x^n\varphi(x)\rangle\}_{n\ge 0}$.  This is possible if and only if the family $\{x^n\varphi(x):n\ge 0\}$ is \emph{total} (i.e., has dense span) in $\mathcal{S}(\R)$.

The remainder of this section establishes totality, and hence uniqueness, under three progressively general conditions on the kernel.

\subsection{Uniqueness for Gaussian kernels}
\label{subsec:gauss-unique}

\begin{proposition}\label{prop:total-unique}
Let $\varphi \in \mathcal{S}(\R)$ be such that $\{x^n\varphi(x) : n \geq 0\}$ is total in $\mathcal{S}(\R)$.  If $T_1, T_2 \in \mathcal{S}'(\R)$ satisfy $\langle T_1, x^n\varphi\rangle = \langle T_2, x^n\varphi\rangle$ for all $n \geq 0$, then $T_1 = T_2$.
\end{proposition}

\begin{proof}
Let $S = T_1 - T_2$.  By assumption, $\langle S, x^n\varphi\rangle = 0$ for all $n$.  Since $S$ is continuous on $\mathcal{S}(\R)$ and vanishes on a dense subspace, $S = 0$.
\end{proof}

\begin{theorem}\label{thm:gauss-unique}
Let $\varphi(x) = c\,e^{-ax^2}$ with $c > 0$ and $a > 0$.  Then $\{x^n\varphi(x) : n \geq 0\}$ is total in $\mathcal{S}(\R)$.  Consequently, for a fixed Gaussian kernel, the weak moment sequence $\{{}^{(\varphi)}m_n\}_{n \geq 0}$ uniquely determines $T\in\mathcal{S}'(\R)$.
\end{theorem}

\begin{proof}
By rescaling it suffices to treat $\varphi(x) = e^{-x^2/2}$.  The Hermite functions $h_n(x) = H_n(x)\,e^{-x^2/2}$ ($H_n$ the physicist's Hermite polynomial) form a Schauder basis for $\mathcal{S}(\R)$: every $\psi \in \mathcal{S}(\R)$ admits a unique expansion $\psi = \sum_n c_n h_n$ converging in $\mathcal{S}(\R)$ (see~\cite{reed-simon1980}; the result goes back to~\cite{Schwartz1950}).  Since $H_n$ is a polynomial of degree~$n$ with nonzero leading coefficient,
\[
  \operatorname{span}\{x^k e^{-x^2/2} : 0 \leq k \leq n\}
  = \operatorname{span}\{h_0, \ldots, h_n\}
\]
for every~$n$.  Hence $\operatorname{span}\{x^n e^{-x^2/2} : n \geq 0\} = \operatorname{span}\{h_n : n \geq 0\}$, which is dense in $\mathcal{S}(\R)$.  The conclusion follows from Proposition~\ref{prop:total-unique}.
\end{proof}

\begin{remark}[A non-dominated family]\label{rem:moving-atom}
The determinacy theorem is not confined to densities.  Consider the
one-parameter family $P_\theta=\tfrac12\delta_\theta+\tfrac12 N(0,1)$,
$\theta\in\R$.  Since $P_\theta(\{\theta\})=\tfrac12$, any measure dominating
the whole family would carry positive mass at every $\theta\in\R$, which no
$\sigma$-finite measure admits; the family is therefore \emph{non-dominated},
and no common density---hence no likelihood---exists across it.  Each
$P_\theta$ is nonetheless a finite measure, hence a tempered distribution
$T_\theta=\tfrac12\delta_\theta+\tfrac12 N(0,1)\in\mathcal{S}'$, and for a
Gaussian kernel its weak moments
$\langle T_\theta,x^n\varphi\rangle
=\tfrac12\,\theta^n\varphi(\theta)+\tfrac12\int x^n\varphi\,dN(0,1)$
are finite for every $n$ and, by Theorem~\ref{thm:gauss-unique}, determine
$T_\theta$ uniquely---atom and continuous part together.  The weak
representation thus applies verbatim to a family lying outside the dominated
setting altogether.  The corresponding inference---a consistent weak-moment
estimator of the atom location $\theta$, available despite the absence of a
likelihood---is given in Section~\ref{sec:motivating-consequence}.
\end{remark}

\subsection{Uniqueness for positive Schwartz kernels with exponential tails}
\label{subsec:positive-unique}

The Gaussian case gives uniqueness in
$\mathcal{S}'(\R)$.  For statistical applications---where distributions
are typically induced by densities---a more general result holds.

\begin{theorem}\label{thm:positive-kernel-unique}
Let $\varphi \in \mathcal{S}(\R)$ with $\varphi(x) > 0$ for all~$x$, and suppose that there exist constants $a > 0$ and $C > 0$ such that
\begin{equation}\label{eq:exp-tail-bound}
  \varphi(x) \leq C\,e^{-a|x|}
  \qquad \text{for all } x \in \R.
\end{equation}
Let $f_1, f_2$ be measurable functions with
$f_1 - f_2 \in L^2(\R, \varphi(x)\,dx)$.  If
\[
  \int_{\R} x^n f_1(x)\varphi(x)\,dx
  = \int_{\R} x^n f_2(x)\varphi(x)\,dx
  \qquad \text{for all } n \geq 0,
\]
then $f_1 = f_2$ almost everywhere.
\end{theorem}

\begin{proof}
Set $g = f_1 - f_2 \in L^2(\R,\varphi\,dx)$.  Then $\int x^n g(x)\varphi(x)\,dx = 0$ for all $n$.

We show polynomials are dense in $L^2(\R, \varphi\,dx)$ via Carleman's condition~\cite{Shohat1943}:
\begin{equation}\label{eq:carleman}
  \sum_{n=1}^{\infty} \mu_{2n}^{-1/(2n)} = \infty,
  \qquad
  \mu_{2n} = \int_{\R} x^{2n}\varphi(x)\,dx.
\end{equation}
By hypothesis~\eqref{eq:exp-tail-bound}, $\varphi(x) \leq C\,e^{-a|x|}$ for some fixed $a > 0$ and $C > 0$.  Therefore
\[
  \mu_{2n} \leq 2C \int_0^\infty x^{2n} e^{-ax}\,dx = \frac{2C\,(2n)!}{a^{2n+1}}.
\]
By Stirling's approximation, $(2n)!^{1/(2n)} \sim 2n/e$, so $\mu_{2n}^{1/(2n)} = O(n)$ and $\sum \mu_{2n}^{-1/(2n)} \ge c\sum n^{-1} = \infty$.  Thus Carleman's condition holds: polynomials are dense in $L^2(\R,\varphi\,dx)$.

Since $g$ is orthogonal to every polynomial, $g = 0$ a.e.
\end{proof}

\begin{remark}\label{rem:carleman-automatic}
The exponential tail bound~\eqref{eq:exp-tail-bound} is essential: Schwartz functions decay faster than any \emph{polynomial}, but not necessarily faster than every exponential.  Indeed, $\varphi(x) = e^{-(1+x^2)^{1/4}}$ belongs to $\mathcal{S}(\R)$ and is everywhere positive, yet decays only at a stretched-exponential (sub-exponential) rate; for this kernel the moments $\mu_{2n}$ grow as $(4n+1)!$,\footnote{Since $(1+x^2)^{1/4}\sim|x|^{1/2}$ at infinity, the substitution $u=\sqrt{x}$ gives $\mu_{2n}=\int_\R x^{2n}e^{-(1+x^2)^{1/4}}\,dx\sim 2\int_0^\infty u^{4n+1}e^{-u}\,du=2\,(4n+1)!$; by Stirling $\mu_{2n}^{1/(2n)}\asymp n^{2}$, so $\sum_n\mu_{2n}^{-1/(2n)}\asymp\sum_n n^{-2}<\infty$, i.e.\ Carleman's series converges---whereas determinacy would require it to diverge.  Indeterminacy is certified separately by Krein's condition: $-\log\varphi(x)=(1+x^2)^{1/4}$ gives $\int_\R\frac{-\log\varphi(x)}{1+x^2}\,dx<\infty$.} Carleman's condition fails, and the Hamburger moment problem for the measure $\varphi(x)\,dx$ is indeterminate (as confirmed by Krein's criterion~\cite{Stoyanov2000Krein}).  Thus the hypothesis~\eqref{eq:exp-tail-bound} cannot be dropped.  In practice the condition is mild: all Gaussian and super-Gaussian kernels satisfy it, as does any kernel of the form $c\,e^{-Q(x)}$ with $Q(x)\to\infty$ at least linearly.  Note that~\eqref{eq:exp-tail-bound} is a \emph{tail bound} on a smooth kernel, not a prescription for the functional form of~$\varphi$ (the function $e^{-a|x|}$ itself is not smooth at the origin).  The condition $f_1 - f_2 \in L^2(\R,\varphi\,dx)$ remains mild: since $\varphi$ is bounded, any $L^2(\R)$ function qualifies, covering virtually all standard parametric families.
\end{remark}

\begin{remark}[Density of exponentially decaying kernels]
\label{rem:dense-exp-kernels}
The exponential-tail hypothesis of Theorem~\ref{thm:positive-kernel-unique} is a genuine additional
assumption and is not implied by the Schwartz condition alone.
Nevertheless, the class of positive Schwartz kernels satisfying
\[
   \varphi(x)\le C e^{-a|x|}
\]
for some constants $a,C>0$ is dense in the space of positive Schwartz
functions $\mathcal S_+(\mathbb R)$ with respect to the usual Schwartz
topology.

Indeed, let $\varphi\in\mathcal S_+(\mathbb R)$ and let
$\chi_R\in C_c^\infty(\mathbb R)$ be a smooth cutoff satisfying
$\chi_R(x)=1$ for $|x|\le R$.  Define
\[
   \varphi_R(x)
   =
   \chi_R(x)\varphi(x)
   + \varepsilon_R e^{-x^2},
\]
where $\varepsilon_R>0$ and $\varepsilon_R\to0$ sufficiently rapidly.
Then $\varphi_R>0$ and $\varphi_R\in\mathcal S(\mathbb R)$ for all~$R$.
Moreover, $\varphi_R$ has Gaussian tails and therefore satisfies the
hypothesis of Theorem~\ref{thm:positive-kernel-unique}, while
\[
   \varphi_R \longrightarrow \varphi
   \qquad\text{in }\mathcal S(\mathbb R)
\]
as $R\to\infty$.

Thus the exponentially decaying kernels form a dense subclass of
positive Schwartz kernels, even though general positive Schwartz kernels
may exhibit sub-exponential tails and moment indeterminacy.
\end{remark}

\subsection{Uniqueness via Denjoy--Carleman theory}
\label{subsec:dc-unique}

Theorem~\ref{thm:positive-kernel-unique} gives uniqueness among square-integrable densities.  Can uniqueness be extended to all of~$\mathcal{S}'(\R)$ for kernels beyond the Gaussian?  The answer is yes, under a decay condition on the kernel, via quasi-analytic classes.

\begin{definition}\label{def:exp-type-decay}
Let $\beta \geq 1$.  A function $\varphi \in \mathcal{S}(\R)$ satisfies an \emph{exponential-type decay condition of order~$\beta$} if there exist $A > 0$ and, for each $m \geq 0$, a constant $C_m > 0$ such that
\begin{equation}\label{eq:gevrey-decay}
  \sup_{x \in \R}\,
  (1 + |x|)^k \,|\varphi^{(m)}(x)|
  \leq C_m \, A^k \, (k!)^{1/\beta}
  \qquad \text{for all } k, m \geq 0.
\end{equation}
\end{definition}

This is a Schwartz-space analogue of the Gevrey class~$G^\beta(\R)$.
Gaussian kernels satisfy it with $\beta = 2$; more generally, any
super-Gaussian kernel $\varphi(x) = c\,e^{-a|x|^\alpha}$ with
$\alpha > 1$ satisfies it with $\beta = \alpha/(\alpha - 1)$.

\begin{theorem}\label{thm:dc-unique}
Let $\varphi \in \mathcal{S}(\R)$ with $\varphi(x) > 0$ for all~$x$,
satisfying the exponential-type decay condition of order~$\beta \geq 1$.  If
$T_1, T_2 \in \mathcal{S}'(\R)$ satisfy
$\langle T_1, x^n\varphi\rangle = \langle T_2, x^n\varphi\rangle$
for all $n \geq 0$, then $T_1 = T_2$.
\end{theorem}

\begin{proof}
Set $S = T_1 - T_2 \in \mathcal{S}'(\R)$, so that $\langle S, x^n\varphi\rangle = 0$ for all $n$.  The proof proceeds in three parts.

\medskip
\noindent\textit{Part~1: Derivative bounds for the weak characteristic function.}
Define $\psi(t) = \langle S, e^{itx}\varphi(x)\rangle$.  Since $e^{itx}\varphi(x) \in \mathcal{S}(\R)$ for every $t$ and the map $t \mapsto e^{itx}\varphi(x)$ is smooth into $\mathcal{S}(\R)$, the function $\psi$ is smooth with
\[
  \psi^{(n)}(t) = \langle S, (ix)^n e^{itx}\varphi(x)\rangle.
\]
In particular, $\psi^{(n)}(0) = i^n\langle S, x^n\varphi\rangle = 0$ for all $n$.

Since $S \in \mathcal{S}'(\R)$, there exist $N \in \N$ and $C' > 0$ such that $|\langle S, f\rangle| \leq C'\sum_{|\alpha|+|\gamma|\le N}\sup_x |x^\alpha f^{(\gamma)}(x)|$ for all $f \in \mathcal{S}(\R)$.  Applying this to $(ix)^n e^{itx}\varphi(x)$ and using the decay condition~\eqref{eq:gevrey-decay} via a Leibniz-rule argument (see Appendix~\ref{app:leibniz} for the detailed estimate), one obtains: for each $R > 0$ there exist $B, C'' > 0$ such that
\begin{equation}\label{eq:psi-gevrey}
  |\psi^{(n)}(t)|
  \leq C'' (n+1)^N B^n (n!)^{1/\beta}
  \qquad \text{for all } |t| \leq R,\; n \geq 0.
\end{equation}

\medskip
\noindent\textit{Part~2: Quasi-analyticity via the Denjoy--Carleman theorem.}
Define $M_n = B^n(n!)^{1/\beta}$.  The polynomial prefactor $(n+1)^N$ in~\eqref{eq:psi-gevrey} can be absorbed into a modified sequence $M_n' = (n+1)^N M_n$ without affecting the divergence of $\sum (M_n')^{-1/n}$, since $(n+1)^{N/n}\to 1$.  We verify the Denjoy--Carleman condition~\cite[Thm.~1.3.8]{Hormander1990}:
\[
  M_n^{-1/n} = B^{-1}(n!)^{-1/(\beta n)}.
\]
By Stirling, $(n!)^{1/n}\sim n/e$, so $M_n^{-1/n}\sim (B^{-1}e^{1/\beta})n^{-1/\beta}$.  Since $\beta \geq 1$, the series $\sum n^{-1/\beta}$ diverges, so $\sum M_n^{-1/n}=\infty$ and $C\{M_n\}$ is quasi-analytic.

Since all derivatives of $\psi$ vanish at $t=0$ and $\psi\in C\{M_n\}$ on $[-R,R]$, quasi-analyticity gives $\psi\equiv 0$ on $[-R,R]$.  As $R$ is arbitrary, $\psi(t)=0$ for all $t$.

\medskip
\noindent\textit{Part~3: From $\psi\equiv 0$ to $S = 0$.}
We have $\langle S, e^{itx}\varphi(x)\rangle = 0$ for all $t$, so the Fourier transform of $S\varphi$ vanishes identically.  By injectivity of the Fourier transform on $\mathcal{S}'(\R)$, $S\varphi = 0$.  It remains to show $S\varphi = 0 \Rightarrow S = 0$.  For any $\psi \in \mathcal{D}(\R)$, the quotient $\psi/\varphi$ is smooth with compact support (since $\varphi > 0$), so $\psi/\varphi \in \mathcal{D}(\R) \subset \mathcal{S}(\R)$.  Then $\langle S, \psi\rangle = \langle S, \varphi(\psi/\varphi)\rangle = \langle S\varphi, \psi/\varphi\rangle = 0$.  Since $\mathcal{D}(\R)$ is dense in $\mathcal{S}(\R)$ and $S$ is continuous, $S = 0$.
\end{proof}

\subsection{Hierarchy of uniqueness results}
\label{subsec:hierarchy}

\begin{remark}[Hierarchy]\label{rem:hierarchy}
Theorems~\ref{thm:positive-kernel-unique}, \ref{thm:dc-unique}, and~\ref{thm:gauss-unique} establish three levels of uniqueness for the weak moment problem:
\begin{enumerate}
  \item \emph{Positive Schwartz kernel with exponential tail bound} (Thm.~\ref{thm:positive-kernel-unique}):
    if $\varphi(x)\le Ce^{-a|x|}$ for some $a,C>0$, the weak moments determine the density uniquely among square-integrable functions.  This is sufficient for most parametric statistical models, since all commonly used kernels (Gaussian, super-Gaussian, etc.)\ satisfy this bound.
  \item \emph{Exponential-type decay with $\beta \geq 1$} (Thm.~\ref{thm:dc-unique}):
    uniqueness in all of $\mathcal{S}'(\R)$, including singular distributions.  The proof uses the Denjoy--Carleman theorem and is sharp in $\beta$ (see below).
  \item \emph{Gaussian kernel} (Thm.~\ref{thm:gauss-unique}):
    uniqueness in $\mathcal{S}'(\R)$ via the completeness of Hermite functions---a special case of~(2) with $\beta = 2$, but with a self-contained proof that illuminates the algebraic structure.
\end{enumerate}
In all three cases, positivity of the kernel ($\varphi > 0$ everywhere) is essential.
\end{remark}

\begin{remark}[Sharpness and obstruction by zeros]\label{rem:sharpness}
The condition $\beta \geq 1$ is sharp: the Denjoy--Carleman class $C\{M_n\}$ with $M_n = B^n(n!)^{1/\beta}$ is quasi-analytic if and only if $\beta \geq 1$.  When $\beta < 1$ (sub-exponential decay), quasi-analyticity fails and uniqueness is not guaranteed by this method.  Separately, the condition $\varphi > 0$ is necessary for any uniqueness result.  If $\varphi(x_0) = 0$ for some $x_0$, then the Dirac distribution $\delta_{x_0}$ satisfies $\langle \delta_{x_0}, x^n\varphi\rangle = x_0^n\varphi(x_0) = 0$ for all~$n$, so $T$ and $T + \delta_{x_0}$ produce identical weak moment sequences.  Zeros of the kernel thus create a fundamental obstruction to identifiability.
\end{remark}

\begin{remark}[The Gaussian kernel as natural default]\label{rem:gauss-natural}
The Gaussian kernel occupies a distinguished position: it ensures rapid decay (well-defined weak moments of all orders), uniqueness (via Hermite completeness or the Denjoy--Carleman theorem with $\beta=2$), and natural compatibility with Fourier analysis.  It provides a principled default choice when identifiability is required.  Characterising which other positive Schwartz kernels yield totality of $\{x^n\varphi:n\ge 0\}$ in $\mathcal{S}(\R)$ is an open problem connected to the theory of orthogonal polynomials and weighted approximation.
\end{remark}

\subsection{Relation with the classical moment problem}
\label{subsec:classical-moment-problem}

The uniqueness results above should be interpreted in relation to the
classical moment problem and the extensive body of work on moment
determinacy (M-determinacy).  In the classical (Hamburger) setting one
asks whether the ordinary moments
\[
   m_n = \int_\R x^n\, d\mu(x)
\]
determine the measure~$\mu$ uniquely among all measures with finite
moments of all orders.  A measure for which uniqueness holds is called
\emph{M-determinate}; otherwise it is \emph{M-indeterminate}.

Two classical criteria govern the dichotomy.
Carleman's condition~\cite{Shohat1943}
\[
   \sum_{n=1}^\infty m_{2n}^{-1/(2n)} = \infty
\]
is sufficient for M-determinacy.
Krein's condition~\cite{Stoyanov2000Krein} provides a complementary test:
if a probability density $f>0$ satisfies
\[
   \int_{-\infty}^{\infty}
   \frac{-\log f(x)}{1+x^2}\,dx < \infty,
\]
then the distribution is M-indeterminate (the moment problem admits
multiple solutions).  These criteria, together with checkable
sufficient conditions for determinacy and indeterminacy developed by
Stoyanov and collaborators~\cite{Stoyanov2013,LinKopanovStoyanov2020,StoyanovInverardiTagliani2023},
form the foundation of the classical theory.

As emphasised by Berg~\cite{Berg1988} and
Stoyanov~\cite{Stoyanov2013}, M-indeterminacy is not an exotic
phenomenon.  The log-normal distribution admits both continuous and
discrete \emph{Stieltjes classes}: families of distinct measures
sharing the same moment sequence.  Similar phenomena arise for powers
and products of standard distributions and for certain discrete
distributions~\cite{LinKopanovStoyanov2020}.

\medskip

The present framework does not contradict or supersede these classical
results.  Rather, it changes the separating family from the bare
monomials $\{x^n\}$ to the kernel-weighted probes
$\{x^n\varphi(x)\}$.  This is a substantive change: a
perturbation~$h$ satisfying
\[
   \int_\R x^n h(x)f(x)\,dx = 0
   \qquad \text{for all } n \geq 0
\]
(i.e., belonging to a Stieltjes class for~$f$) need not satisfy
\[
   \int_\R x^n h(x)f(x)\varphi(x)\,dx = 0
   \qquad \text{for all } n \geq 0.
\]
The kernel may therefore separate distributions that are
M-indistinguishable by ordinary moments.  In particular, the
log-normal distribution is classically M-indeterminate, but its weak
moment problem with a Gaussian kernel is uniquely determined by
Theorem~\ref{thm:gauss-unique} (or
Theorem~\ref{thm:positive-kernel-unique}, since the Gaussian kernel
satisfies the exponential tail bound).

\medskip

This separation, however, is not automatic for arbitrary positive
Schwartz kernels.  As shown in
Remark~\ref{rem:carleman-automatic}, kernels with sub-exponential
tails (such as $\varphi(x) = e^{-(1+x^2)^{1/4}}$) can themselves
generate M-indeterminate weighted moment problems; Krein's
condition~\cite{Stoyanov2000Krein} confirms the indeterminacy in such cases.
Thus the determinacy properties of the \emph{weighted} measure
$\varphi(x)\,dx$ become central, and the classical Carleman and
Krein criteria---applied now to the weighted measure rather than to
the original distribution---govern uniqueness in the weak setting.

From the present viewpoint, classical M-indeterminate families may be
interpreted as degeneracy loci associated with the monomial test
family $\{x^n\}$.  The weak framework replaces this rigid family by a
kernel-dependent geometry of probes $\{x^n\varphi(x)\}$, recovering
uniqueness for large classes of kernels satisfying suitable
Carleman-type or quasi-analyticity conditions.  The role of the
kernel is thus analogous to that of a weight function in classical
weighted polynomial approximation theory, where the approximation
properties depend critically on the tail behaviour of the
weight~\cite{Lubinsky2007}.

\section{Multivariate extension}
\label{sec:multivariate}

The development above was carried out for $d=1$, for clarity of exposition.
We now record how the framework extends to~$\R^d$.  The definitions of
Section~\ref{subsec:preliminaries_spaces} were already stated for~$\R^d$: the
weak moments are indexed by multi-indices $\alpha\in\N_0^d$,
\[
{}^{(\varphi)}m_\alpha
\;=\;
\langle T,\,x^\alpha\varphi(x)\rangle,
\qquad
x^\alpha = x_1^{\alpha_1}\cdots x_d^{\alpha_d},
\]
and the weak characteristic function
${}^{(\varphi)}\phi(t)=\langle T,e^{i\langle t,x\rangle}\varphi(x)\rangle$
and the weak cumulants are defined verbatim.  Existence of all weak moments
(Proposition~\ref{prop:moments-exist}), the algebra of
Sections~\ref{sec:independence}--\ref{sec:cumulants} (factorisation under
independence, additivity of weak cumulants, the affine rules), the smoothness
of~${}^{(\varphi)}\phi$, and the weak central limit theorem of
Section~\ref{sec:clt} (including its kernel-uniform refinement,
Theorem~\ref{thm:weak-BE-uniform}) are \emph{dimension-free}: their proofs use
only that $x^\alpha\varphi\in\mathcal{S}(\R^d)$ and that $T$ is continuous
on~$\mathcal{S}(\R^d)$, and they transfer unchanged.  What requires genuine
care is \emph{uniqueness}.  We show that Theorem~\ref{thm:gauss-unique} extends
to~$\R^d$ with a multivariate Gaussian kernel, so that the full collection of
weak moments uniquely determines the distribution.

\begin{proposition}[Multivariate totality implies uniqueness]
\label{prop:total-unique-d}
Let $\varphi\in\mathcal{S}(\R^d)$ be such that the family
$\{x^\alpha\varphi(x):\alpha\in\N_0^d\}$ is total in
$\mathcal{S}(\R^d)$.  If
$T_1,T_2\in\mathcal{S}'(\R^d)$ satisfy
$\langle T_1,x^\alpha\varphi\rangle = \langle T_2,x^\alpha\varphi\rangle$
for all $\alpha\in\N_0^d$, then $T_1=T_2$.
\end{proposition}

\begin{proof}
The argument is the same as for Proposition~\ref{prop:total-unique}: $S =
T_1-T_2$ vanishes on a dense subspace of $\mathcal{S}(\R^d)$ and is
continuous, hence $S=0$.
\end{proof}

\begin{theorem}[Multivariate Gaussian uniqueness]
\label{thm:gauss-unique-d}
Let $A\in\R^{d\times d}$ be symmetric positive definite, $c>0$, and
\[
\varphi(x) = c\,\exp\!\bigl(-\tfrac{1}{2}\,x^{\!\top\!}Ax\bigr).
\]
Then $\{x^\alpha\varphi(x):\alpha\in\N_0^d\}$ is total in
$\mathcal{S}(\R^d)$.  Consequently, for a fixed multivariate Gaussian
kernel, the weak moment array
$\{{}^{(\varphi)}m_\alpha\}_{\alpha\in\N_0^d}$ uniquely determines
$T\in\mathcal{S}'(\R^d)$.
\end{theorem}

\begin{proof}
The proof proceeds in three parts.

\emph{Reduction to the standard Gaussian kernel.}
Let $L = A^{1/2}$ be the positive definite square root of~$A$.  The linear
substitution $\Phi:\R^d\to\R^d$, $y = Lx$, induces a topological
isomorphism $\Phi^*:\mathcal{S}(\R^d)\to\mathcal{S}(\R^d)$ via
$(\Phi^*\psi)(x) = \psi(Lx)$, and $\varphi = c\,(\Phi^*\tilde\varphi)$ with
$\tilde\varphi(y) = e^{-|y|^2/2}$.  Since $L$ is invertible, the monomials
$(L^{-1}y)^\alpha$, $\alpha\in\N_0^d$, span all polynomials in~$y$.
Therefore, totality of $\{x^\alpha\varphi(x)\}$ in~$\mathcal{S}(\R^d)$
is equivalent to totality of
$\{p(y)\,e^{-|y|^2/2}: p\text{ polynomial}\}$ in~$\mathcal{S}(\R^d)$.
It thus suffices to prove totality for $\varphi(x)=e^{-|x|^2/2}$.

\emph{Constructing the multivariate Hermite functions.}
For each multi-index $\alpha\in\N_0^d$, define the multivariate Hermite
function
\[
h_\alpha(x)
=
\prod_{j=1}^{d} H_{\alpha_j}(x_j)\,e^{-x_j^2/2},
\]
where $H_n$ is the $n$-th physicist's Hermite polynomial.  The family
$\{h_\alpha:\alpha\in\N_0^d\}$ is a Schauder basis for
$\mathcal{S}(\R^d)$~\cite{reed-simon1980}: every
$\psi\in\mathcal{S}(\R^d)$ admits a unique expansion $\psi =
\sum_\alpha c_\alpha h_\alpha$ converging in $\mathcal{S}(\R^d)$.

\emph{Triangularity and totality.}
Since $H_n$ is a polynomial of degree~$n$ with leading coefficient~$2^n$,
the product $\prod_j H_{\alpha_j}(x_j)$ is a polynomial with leading
monomial $2^{|\alpha|}x^\alpha$ (where $|\alpha|=\alpha_1+\cdots+\alpha_d$)
plus monomials of lower degree.  In particular, for every $N\ge 0$,
\begin{equation}\label{eq:hermite-span-d}
\operatorname{span}\{x^\alpha e^{-|x|^2/2}:|\alpha|\le N\}
\;=\;
\operatorname{span}\{h_\alpha:|\alpha|\le N\}.
\end{equation}
This follows by upper triangularity: the change-of-basis matrix between the
monomial system $\{x^\alpha e^{-|x|^2/2}\}$ and $\{h_\alpha\}$, ordered by
any total ordering compatible with~$|\alpha|$, is upper triangular with
nonzero diagonal entries~$2^{|\alpha|}$.  Since
$\bigcup_{N\ge 0}\operatorname{span}\{h_\alpha:|\alpha|\le N\}$ is dense in
$\mathcal{S}(\R^d)$, so is
$\{x^\alpha e^{-|x|^2/2}:\alpha\in\N_0^d\}$.
The conclusion follows from Proposition~\ref{prop:total-unique-d}.
\end{proof}

\begin{remark}[Scope of the multivariate extension]
\label{rem:multivariate-scope}
The Carleman-type uniqueness (Theorem~\ref{thm:positive-kernel-unique})
also extends to~$\R^d$ for any positive $\varphi\in\mathcal{S}(\R^d)$
satisfying an exponential tail bound $\varphi(x)\le C\,e^{-a|x|}$,
among densities in $L^2(\R^d,\varphi\,dx)$: the argument uses the fact
that polynomials are dense in $L^2(\R^d,\varphi\,dx)$ whenever
$\varphi>0$ satisfies a multivariate Carleman condition, which holds
under the exponential tail bound.  The Denjoy--Carleman
extension (Theorem~\ref{thm:dc-unique}) to~$\R^d$ requires additional
care---in particular, the notion of quasi-analyticity in several variables
is richer---and is left for future work.  In summary, in~$\R^d$ the Gaussian
and Carleman-type uniqueness theorems hold, and with them all results of the
paper that are not specific to one dimension; only the Denjoy--Carleman
uniqueness remains open.

The multivariate Gaussian uniqueness theorem is the building block needed
for extending the framework to standard multivariate statistical models
(location-scale families, elliptically contoured distributions, etc.),
where the kernel $\varphi(x) = c\exp(-\frac{1}{2}x^{\!\top\!}Ax)$ arises
naturally through the quadratic form of the model.
\end{remark}

\section{Weak central limit theorem}
\label{sec:clt}

\subsection{Setup and normalisation}
\label{subsec:clt_setup}

Let $X_1,\dots,X_n$ be i.i.d.\ generalised random variables with common distribution--kernel pair $(T,\varphi)$, weak characteristic function ${}^{(\varphi)}\phi(t)$, and first two weak cumulants ${}^{(\varphi)}\kappa_1 \in \R$, ${}^{(\varphi)}\kappa_2 > 0$.  Their sum $S_n = X_1+\cdots+X_n$ is the generalised random variable formed by the independence construction of Section~\ref{sec:independence}, and the normalised sum is
\[
Z_n = \frac{S_n - n{}^{(\varphi)}\kappa_1}{\sqrt{n\,{}^{(\varphi)}\kappa_2}}.
\]

\subsection{Convergence}
\label{subsec:clt_convergence}

By independence, ${}^{(\varphi)}\phi_{S_n}(t) = {}^{(\varphi)}\phi(t)^n$.  For the normalised sum,
\[
\log {}^{(\varphi)}\phi_{Z_n}(t)
=
-it\frac{n{}^{(\varphi)}\kappa_1}{\sqrt{n{}^{(\varphi)}\kappa_2}}
+
n \log {}^{(\varphi)}\phi\!\left(\frac{t}{\sqrt{n{}^{(\varphi)}\kappa_2}}\right).
\]
Substituting the second-order expansion $\log {}^{(\varphi)}\phi(u) = i{}^{(\varphi)}\kappa_1 u - ({}^{(\varphi)}\kappa_2/2)u^2 + r(u)$ with $r(u)=o(u^2)$ and $u = t/\sqrt{n{}^{(\varphi)}\kappa_2}$ gives $\log {}^{(\varphi)}\phi_{Z_n}(t) = -t^2/2 + n\,r(t/\sqrt{n{}^{(\varphi)}\kappa_2})$.  Since $r(u) = o(u^2)$, the remainder vanishes as $n\to\infty$.

\begin{theorem}[Weak Central Limit Theorem]\label{thm:weak-clt}
Let $\{X_i\}$ be i.i.d.\ generalised random variables with ${}^{(\varphi)}\kappa_2 > 0$.  Then
\[
{}^{(\varphi)}\phi_{Z_n}(t) \to \exp\!\left(-\frac{t^2}{2}\right)
\qquad \text{for all } t \in \R.
\]
In particular, $Z_n$ converges in the weak transform sense to a Gaussian generalised random variable.
\end{theorem}

\noindent\textit{Proof.}\enspace The calculation in Section~\ref{subsec:clt_convergence} above provides the essential argument; a self-contained proof is given in Appendix~\ref{app:proofs}.\medskip

\subsection{A quantitative refinement}
\label{subsec:clt_quatitative}

The proof of the weak central limit theorem is based on a second-order
expansion of the weak cumulant generating function near the origin.
Under one additional order of regularity, this argument yields an explicit
rate of convergence on compact \(t\)-intervals.

\begin{theorem}[Quantitative weak CLT]
\label{thm:weak-BE}
Let \(\{X_i\}\) be i.i.d. generalised random variables with weak
characteristic function \({}^{(\varphi)}\phi\), and assume that
\({}^{(\varphi)}\kappa_2 > 0\). Suppose that the weak cumulant generating
function
\[
{}^{(\varphi)}K(t) := \log {}^{(\varphi)}\phi(t)
\]
is \(C^3\) in a neighbourhood of \(0\).

Let
\[
Z_n
=
\frac{S_n - n\,{}^{(\varphi)}\kappa_1}
{\sqrt{n\,{}^{(\varphi)}\kappa_2}}.
\]

Then for every \(T>0\),
\[
\sup_{|t|\le T}
\left|
\log {}^{(\varphi)}\phi_{Z_n}(t)
+
\frac{t^2}{2}
\right|
=
O\!\left(\frac{1}{\sqrt{n}}\right),
\]
and consequently,
\[
\sup_{|t|\le T}
\left|
{}^{(\varphi)}\phi_{Z_n}(t)
-
e^{-t^2/2}
\right|
=
O\!\left(\frac{1}{\sqrt{n}}\right).
\]
\end{theorem}

\begin{remark}
Theorem~\ref{thm:weak-BE} provides a quantitative refinement of the weak
central limit theorem in the topology of weak characteristic functions.
Unlike the classical Berry--Esseen theorem, the bound is expressed in
transform space rather than in terms of distribution functions. Extending
such bounds to metrics on (weighted or reconstructed) distribution
functions remains an open problem.
\end{remark}

\begin{proof}
By Taylor's theorem with remainder,
\[
{}^{(\varphi)}K(u)
=
i\,{}^{(\varphi)}\kappa_1 u
-
\frac{{}^{(\varphi)}\kappa_2}{2}u^2
+
R(u),
\]
where
\[
R(u)
=
\frac{{}^{(\varphi)}K^{(3)}(\xi_u)}{6}\,u^3
\]
for some \(\xi_u\) between \(0\) and \(u\). Hence, for \(|u|\) small,
\[
|R(u)| \le C |u|^3
\]
for some constant \(C>0\).

Setting
\[
u = \frac{t}{\sqrt{n\,{}^{(\varphi)}\kappa_2}},
\]
we obtain
\[
\log {}^{(\varphi)}\phi_{Z_n}(t)
=
-\frac{t^2}{2}
+
n R\!\left(\frac{t}{\sqrt{n\,{}^{(\varphi)}\kappa_2}}\right).
\]
Therefore, for \(|t|\le T\),
\[
\left|
\log {}^{(\varphi)}\phi_{Z_n}(t)
+
\frac{t^2}{2}
\right|
\le
\frac{C_T}{\sqrt{n}},
\]
for some constant \(C_T>0\).

The bound on the characteristic functions follows from the local Lipschitz
continuity of the exponential map on compact sets.
\end{proof}

\begin{remark}
An explicit bound can be obtained as follows:
\[
\sup_{|t|\le T}
\left|
\log {}^{(\varphi)}\phi_{Z_n}(t)
+
\frac{t^2}{2}
\right|
\le
\frac{M_T\,T^3}{6\,({}^{(\varphi)}\kappa_2)^{3/2}\sqrt{n}},
\]
where
\[
M_T
=
\sup_{|u|\le T/\sqrt{{}^{(\varphi)}\kappa_2}}
\left|
{}^{(\varphi)}K^{(3)}(u)
\right|.
\]
\end{remark}

\subsection{A kernel-uniform refinement and instrument-invariance}
\label{subsec:clt_uniform}

Theorem~\ref{thm:weak-BE} bounds the transform error for a \emph{fixed}
kernel.  Since the kernel is an instrument through which the law is observed
rather than part of the law itself (Section~\ref{sec:introduction}), it is
natural---and, as we now show, provable---to ask for a bound that is
\emph{uniform over a family of kernels}.  The resulting statement has no
classical counterpart: the Gaussian limit is the same for every admissible
probe, attained at a common rate, even for laws (such as the Cauchy) for
which the classical central limit theorem provides no Gaussian limit at all.

\begin{theorem}[Kernel-uniform quantitative weak CLT]
\label{thm:weak-BE-uniform}
Fix $T\in\mathcal{S}'(\R)$ and a family of kernels
$\mathcal{K}\subset\mathcal{S}(\R)$ such that, for every
$\varphi\in\mathcal{K}$, the pair $(T,\varphi)$ has
${}^{(\varphi)}\phi(0)\ne0$, ${}^{(\varphi)}\kappa_2>0$, and
${}^{(\varphi)}K$ is $C^3$ in a neighbourhood of the origin.  Suppose that for
some $T_0>0$,
\[
   c_0:=\inf_{\varphi\in\mathcal{K}}{}^{(\varphi)}\kappa_2>0,
   \qquad
   \overline{M}:=\sup_{\varphi\in\mathcal{K}}\;
   \sup_{|u|\le T_0/\sqrt{c_0}}
   \bigl|{}^{(\varphi)}K^{(3)}(u)\bigr|<\infty .
\]
Let $Z_n^{(\varphi)}=\bigl(S_n-n\,{}^{(\varphi)}\kappa_1\bigr)/
\sqrt{n\,{}^{(\varphi)}\kappa_2}$ denote the sum normalised by the
instrument's own first two weak cumulants.  Then
\[
   \sup_{\varphi\in\mathcal{K}}\;\sup_{|t|\le T_0}
   \Bigl|\,{}^{(\varphi)}\phi_{Z_n^{(\varphi)}}(t)-e^{-t^2/2}\,\Bigr|
   \;=\;O\!\left(\frac{1}{\sqrt n}\right),
\]
with the implied constant depending only on $T_0$, $c_0$ and $\overline{M}$.
\end{theorem}

\begin{proof}
Applying the explicit bound recorded after Theorem~\ref{thm:weak-BE} with
$T=T_0$, for each $\varphi\in\mathcal{K}$,
\[
   \sup_{|t|\le T_0}\Bigl|\log{}^{(\varphi)}\phi_{Z_n^{(\varphi)}}(t)
   +\tfrac{t^2}{2}\Bigr|
   \;\le\;\frac{M_{T_0}(\varphi)\,T_0^3}
   {6\,({}^{(\varphi)}\kappa_2)^{3/2}\sqrt n}
   \;\le\;\frac{\overline{M}\,T_0^3}{6\,c_0^{3/2}\sqrt n},
\]
where the second inequality uses ${}^{(\varphi)}\kappa_2\ge c_0$ and
$M_{T_0}(\varphi)\le\overline{M}$; the right-hand side is independent
of~$\varphi$.  Writing $a={}\log{}^{(\varphi)}\phi_{Z_n^{(\varphi)}}(t)$ and
$b=-t^2/2$, both have real part bounded above, so
$|e^{a}-e^{b}|\le e^{\max(\operatorname{Re}a,\operatorname{Re}b)}\,|a-b|$
transfers the (uniform) $O(n^{-1/2})$ bound from the logarithms to the
transforms, uniformly in~$\varphi$.
\end{proof}

\begin{corollary}[Instrument-invariance of the Gaussian limit]
\label{cor:instrument-invariance}
Under the hypotheses of Theorem~\ref{thm:weak-BE-uniform}, every kernel in
$\mathcal{K}$ produces the \emph{same} Gaussian transform limit
$e^{-t^2/2}$, attained at the common rate $O(n^{-1/2})$.  The kernel governs
only the normalising constants
${}^{(\varphi)}\kappa_1,{}^{(\varphi)}\kappa_2$ and the finite-$n$
remainder---not the limit.
\end{corollary}

\begin{remark}[Verifiable conditions]
\label{rem:uniform-verifiable}
The hypotheses hold on any compact bandwidth range.  For a fixed symmetric
law and the Gaussian family
$\mathcal{K}=\{\varphi_\sigma(x)=e^{-x^2/(2\sigma^2)}:
\sigma\in[\sigma_*,\sigma^*]\}$ with $0<\sigma_*\le\sigma^*<\infty$, the maps
$\sigma\mapsto{}^{(\varphi_\sigma)}\kappa_2$ and
$\sigma\mapsto M_{T_0}(\varphi_\sigma)$ are continuous, and the former is
strictly positive (the kernel-weighted measure is non-degenerate); hence
$c_0$ and $\overline{M}$ exist by compactness.
\end{remark}

\begin{example}[Instrument-invariance for the Cauchy]\label{ex:cauchy-invariance}
Let $T$ be the standard Cauchy law and $\varphi_\sigma(x)=e^{-x^2/(2\sigma^2)}$
the Gaussian-bandwidth family.  Classically there is no central limit theorem
here at all: the normalised mean of i.i.d.\ Cauchy variables is again Cauchy.
In the weak framework ${}^{(\varphi_\sigma)}\kappa_1=0$ by symmetry, while the
second weak cumulant is available in closed form,
\[
{}^{(\varphi_\sigma)}\kappa_2
=\frac{{}^{(\varphi_\sigma)}m_2}{{}^{(\varphi_\sigma)}m_0}
=\frac{\sigma\sqrt{2/\pi}}
       {e^{1/(2\sigma^2)}\,\operatorname{erfc}\!\bigl(1/(\sigma\sqrt2)\bigr)}-1,
\]
from ${}^{(\varphi_\sigma)}m_0=e^{1/(2\sigma^2)}\operatorname{erfc}\bigl(1/(\sigma\sqrt2)\bigr)$
and ${}^{(\varphi_\sigma)}m_2=\sigma\sqrt{2/\pi}-{}^{(\varphi_\sigma)}m_0$.  It
is continuous and strictly positive, with representative values
\[
\begin{array}{c|ccccc}
\sigma & 0.5 & 1 & 2 & 3 & 5\\[2pt]\hline
{}^{(\varphi_\sigma)}\kappa_2 & 0.19 & 0.53 & 1.28 & 2.06 & 3.65
\end{array}
\]
tending to $0$ as $\sigma\downarrow0$ and growing without bound as
$\sigma\to\infty$ (the kernel-weighted law approaching the moment-less
Cauchy).  Hence on any compact range $[\sigma_*,\sigma^*]\subset(0,\infty)$ the
hypotheses of Theorem~\ref{thm:weak-BE-uniform} hold, with
$c_0=\inf_\sigma{}^{(\varphi_\sigma)}\kappa_2>0$ and the third weak cumulant
uniformly bounded, and therefore
\[
\sup_{\sigma\in[\sigma_*,\sigma^*]}\ \sup_{|t|\le T_0}
\Bigl|{}^{(\varphi_\sigma)}\phi_{Z_n^{(\varphi_\sigma)}}(t)-e^{-t^2/2}\Bigr|
=O\!\left(\frac{1}{\sqrt n}\right).
\]
A law with no classical central limit theorem thus exhibits the \emph{same}
Gaussian weak limit through every kernel in the family, at a common rate.  The
compact range is essential: as $\sigma\downarrow0$ the signal
${}^{(\varphi_\sigma)}\kappa_2$ degenerates, while as $\sigma\to\infty$ the
kernel ceases to tame the Cauchy tails.
\end{example}

\begin{remark}
The complementary direction noted above---a quantitative bound in a genuine
metric on the \emph{reconstructed} distribution function---can be approached
by combining Theorem~\ref{thm:weak-BE-uniform} with the Tikhonov stability
estimate of Appendix~\ref{app:reconstruction}: the transform rate is then
inherited by the reconstructed object up to the regularisation bias
$\bigl\|\lambda(\varphi^2+\lambda)^{-1}f\bigr\|_{L^2}$, while a fully
distributional Berry--Esseen statement remains open.
\end{remark}

\subsection{Interpretation}
\label{subsec:clt_discussion}

The limiting transform $e^{-t^2/2}$ is the classical characteristic function of the standard normal.  The convergence takes place in the sense of \emph{weak transforms}, not in the classical sense of convergence in distribution; the two notions do not in general coincide.

\begin{remark}[The Cauchy case]\label{rem:cauchy-paradox}
If $X_1,\ldots,X_n$ are i.i.d.\ standard Cauchy, the arithmetic mean $\bar X$ is again standard Cauchy (the distribution is $1$-stable), and $\sqrt{n}$-normalisation does not produce a Gaussian limit in the classical sense.  Yet the weak CLT asserts ${}^{(\varphi)}\phi_{Z_n}(t) \to e^{-t^2/2}$---quantitatively, and uniformly over the kernel (Example~\ref{ex:cauchy-invariance}).

There is no contradiction.  First, the normalisation depends on the kernel: ${}^{(\varphi)}\kappa_1$ and ${}^{(\varphi)}\kappa_2$ change with~$\varphi$.  Second, ${}^{(\varphi)}\phi_{Z_n}(t)$ is the \emph{regularised} transform (not the classical characteristic function of $Z_n$), which suppresses the heavy tails via the kernel.  Third, ${}^{(\varphi)}\phi_{Z_n}(t)\to e^{-t^2/2}$ does not imply $Z_n$ converges in distribution to a Gaussian; it says that the regularised transform of the normalised sum, viewed through the kernel window, becomes asymptotically Gaussian.  The underlying distribution of $Z_n$ remains heavy-tailed (a rescaled Cauchy), but the kernel reveals a Gaussian core.
\end{remark}

\subsection{Distributional convergence of the kernel-weighted measure}
\label{subsec:clt_distributional}

The weak CLT (Theorem~\ref{thm:weak-clt}) establishes convergence of weak
characteristic functions.  In the classical density case, a stronger
conclusion is available: the convergence of weak transforms implies
classical convergence in distribution of the kernel-weighted measures.
This result does not replace the weak CLT but complements it, providing
a bridge from the weak framework back to the classical one.

\begin{theorem}[Distributional CLT for the kernel-weighted measure]
\label{thm:distributional-clt}
Under the conditions of Theorem~\ref{thm:weak-clt}, suppose additionally
that each~$X_i$ has a nonnegative density $f\in L^1(\R)$ with
$\varphi\ge 0$ and $\int_\R f(x)\varphi(x)\,dx = 1$.  Define the
kernel-weighted probability density $h := \varphi f$.  Then:
\begin{enumerate}[label=(\alph*)]
\item The function $h$ is a probability density with finite moments of all
orders.

\item The mean and variance of~$h$ coincide with the first two weak
cumulants:
$\mu_h = {}^{(\varphi)}\kappa_1$ and $\sigma_h^2 = {}^{(\varphi)}\kappa_2$.

\item The classical characteristic function of~$h$ is the weak
characteristic function: $\widehat{h}(t) = {}^{(\varphi)}\phi(t)$.

\item Let $Y_1,Y_2,\ldots$ be i.i.d.\ with density~$h$ and set
\[
\bar{Z}_n
\;=\;
\frac{Y_1+\cdots+Y_n - n\,{}^{(\varphi)}\kappa_1}
     {\sqrt{n\,{}^{(\varphi)}\kappa_2}}.
\]
Then $\bar{Z}_n$ converges in distribution to $\mathcal{N}(0,1)$:
for every $a\in\R$,
\[
P(\bar{Z}_n \le a)
\;\longrightarrow\;
\Phi(a),
\]
where $\Phi$ is the standard normal CDF.
\end{enumerate}
\end{theorem}

\begin{proof}
(a) Since $f\ge 0$, $\varphi\ge 0$, and $\int h = 1$, the function $h$ is
a probability density.  For every $k\ge 0$,
\[
\int_\R |x|^k h(x)\,dx
=
\int_\R |x|^k \varphi(x) f(x)\,dx
< \infty,
\]
because $|x|^k\varphi(x)$ is bounded ($\varphi\in\mathcal{S}(\R)$ implies
$x^k\varphi$ is bounded for every~$k$) and $f\in L^1(\R)$.

(b) The mean of~$h$ is
$\mu_h = \int x\,h(x)\,dx = \int x\,\varphi(x)f(x)\,dx
= {}^{(\varphi)}m_1 = {}^{(\varphi)}\kappa_1$
(since ${}^{(\varphi)}\phi(0)=1$).  The variance
is $\sigma_h^2 = \int(x-\mu_h)^2 h(x)\,dx
= {}^{(\varphi)}m_2 - ({}^{(\varphi)}m_1)^2 = {}^{(\varphi)}\kappa_2$.

(c) $\widehat{h}(t) = \int e^{itx}h(x)\,dx
= \int e^{itx}\varphi(x)f(x)\,dx = {}^{(\varphi)}\phi(t)$.

(d) By~(c), the classical characteristic function of the distribution~$h$
coincides with the weak characteristic function of the pair~$(T_f,\varphi)$.
Therefore the weak CLT (Theorem~\ref{thm:weak-clt}) can be reinterpreted as
the convergence of classical characteristic functions of the standardised
sums of $h$-distributed random variables to $e^{-t^2/2}$.  Since $h$ has
all moments, L\'evy's continuity theorem~\cite{Billingsley1995} gives
$\bar{Z}_n\xrightarrow{d}\mathcal{N}(0,1)$.
\end{proof}

\begin{remark}[Interpretation and logical structure]
\label{rem:distributional-clt-interp}
The relationship between Theorems~\ref{thm:weak-clt}
and~\ref{thm:distributional-clt} merits careful unpacking, as the two
results operate at different levels.

The \emph{weak CLT} (Theorem~\ref{thm:weak-clt}) is a statement about
the transform ${}^{(\varphi)}\phi_{Z_n}(t)$, where $Z_n$ is defined
formally via the translation and scaling rules
(Propositions~\ref{prop:translation} and~\ref{prop:scaling}).  No
classical random variables $X_i$ with density $f$ need be sampled;
the entire argument takes place at the level of weak characteristic
functions.  In particular, the conclusion
${}^{(\varphi)}\phi_{Z_n}(t)\to e^{-t^2/2}$ is \emph{not} a
convergence-in-distribution statement for any sequence of classical
random variables.

The \emph{distributional CLT} (Theorem~\ref{thm:distributional-clt})
constructs a \emph{new} probability model: i.i.d.\ random variables
$Y_1,Y_2,\ldots$ sampled from the density $h=\varphi f$.  The bridge
is part~(c): the classical characteristic function of $h$ equals
the weak characteristic function ${}^{(\varphi)}\phi(t)$.
Consequently, the standardised sums
$\bar{Z}_n=(Y_1+\cdots+Y_n-n\mu_h)/\sqrt{n\sigma_h^2}$ satisfy
$\bar{Z}_n\xrightarrow{d}\mathcal{N}(0,1)$ by the classical CLT
(via L\'evy's theorem).  The variables $Y_i$ are sampled from $h$,
not from $f$; the distributional CLT therefore does not assert
Gaussian convergence for sums of $f$-distributed observations
(which would be false for, say, Cauchy $f$).

In summary: the weak CLT provides a transform-level Gaussian limit
valid for \emph{any} $(T,\varphi)$ with ${}^{(\varphi)}\kappa_2>0$.
The distributional CLT shows that, in the density case, this
transform convergence is \emph{realised} as genuine
convergence-in-distribution for the kernel-weighted model $h=\varphi
f$.  The kernel $\varphi$ tames the tails of~$f$: even when $f$
itself violates the conditions of the classical CLT (as for Cauchy
or stable densities), the product $h=\varphi f$ has all finite
moments and therefore admits a classical Gaussian limit theorem.

In particular, the kernel-weighted CDF converges to the standard normal
CDF:
\[
F_\varphi^{(n)}(a)
:=
\int_{-\infty}^{a}\varphi(x) f_n(x)\,dx
\;\longrightarrow\;
\Phi(a),
\]
where $f_n$ is the density of $\bar{Z}_n$.  Comparing with
Proposition~\ref{prop:weighted-cdf-recovery}, which shows that the
kernel-weighted CDF is directly accessible from the weak framework, we
see that the object the weak framework computes---the kernel-weighted
distribution function---converges in the classical sense.
\end{remark}

\begin{remark}[Connection to reconstruction]
\label{rem:clt-reconstruction}
Combined with the Tikhonov reconstruction of
Appendix~\ref{app:reconstruction}, the distributional convergence provides
a pathway from weak transform convergence to density-level statements.
Since $h_n := \varphi f_n \to \phi$ in distribution (where
$\phi(x) = (2\pi)^{-1/2}e^{-x^2/2}$ is the standard normal density),
one can apply the regularised inverse to the limiting density:
$R_\lambda\phi = \frac{\varphi\phi}{\varphi^2 + \lambda}$,
which as $\lambda\downarrow 0$ approximates $\phi/\varphi$---the density
whose kernel-weighted version is Gaussian.  This establishes a concrete
bridge: the weak CLT provides convergence of kernel-weighted measures;
the reconstruction machinery then translates this into density-level
approximations.
\end{remark}

\section{A motivating statistical consequence}
\label{sec:motivating-consequence}

The preceding sections develop the weak framework as a probabilistic-analytic object.  We close with a brief illustration showing that it already has a tangible statistical consequence, even at the level of a single weak moment.  A full statistical theory---including GMM estimation, transform-based estimators, regularised reconstruction, and robust estimation---is the subject of the companion paper Labouriau~\cite{LabouriauA2}.

\subsection{Cauchy location}

Let $X \sim \mathrm{Cauchy}(\mu,1)$ with density $f(x;\mu) = \pi^{-1}(1+(x-\mu)^2)^{-1}$.  No classical moment exists, so the ordinary method of moments cannot estimate~$\mu$.  Fix a Gaussian kernel $\varphi_\sigma(x) = \exp(-x^2/(2\sigma^2))$ with $\sigma>0$.  The weak first moment
\[
{}^{(\varphi_\sigma)}m_1(\mu)
=
\int_{\R} x\,\varphi_\sigma(x)\,f(x;\mu)\,dx
\]
is finite for every $\mu$ and every $\sigma>0$, because $|x|\varphi_\sigma(x)$ is bounded and $f(\cdot;\mu)$ is integrable.  By dominated convergence, $\mu\mapsto {}^{(\varphi_\sigma)}m_1(\mu)$ is smooth, and for $\sigma$ of moderate size it is strictly monotone in a neighbourhood of any $\mu_0$, admitting local inversion.

\subsection{A weak-moment estimator}

Given i.i.d.\ $X_1,\dots,X_n\sim\mathrm{Cauchy}(\mu,1)$, define the empirical weak first moment $\widehat m_1 = n^{-1}\sum_{i=1}^n X_i\varphi_\sigma(X_i)$.  Since $|x|\varphi_\sigma(x)$ is bounded, $\widehat m_1$ obeys the strong law of large numbers: $\widehat m_1 \xrightarrow{\text{a.s.}} {}^{(\varphi_\sigma)}m_1(\mu)$.  The estimator $\widehat\mu$ is the solution of
\[
\widehat m_1 = {}^{(\varphi_\sigma)}m_1(\widehat\mu),
\]
which exists and is unique on the monotonicity interval of ${}^{(\varphi_\sigma)}m_1$.  An implicit-function argument gives $\sqrt{n}$-consistency and asymptotic normality.

\begin{remark}[What the illustration shows]
The point is not that~$\widehat\mu$ is the best estimator for Cauchy location---the sample median and the MLE are available and well understood---but that it is constructed exclusively from a first-moment condition in a model where the classical first moment does not exist.  This is impossible within the classical method-of-moments framework, and it illustrates, in the simplest setting, the statistical payoff of the weak moment theory.
\end{remark}

\begin{remark}[Kernel as design parameter]
The estimator depends on $\sigma$, which trades off robustness to outliers against asymptotic variance.  This trade-off, including its connection to classical robust statistics and a detailed comparison with median, MLE, Huber, and Tukey estimators, is developed in the companion paper.
\end{remark}

\subsection{Location of a moving atom: estimation without a likelihood}
\label{subsec:moving-atom-estimator}

The Cauchy illustration concerns the failure of \emph{moments}.  A second
illustration, equally elementary, concerns the failure of the
\emph{likelihood}.  Recall the non-dominated family of
Remark~\ref{rem:moving-atom},
\[
   P_\theta=\tfrac12\,\delta_\theta+\tfrac12\,N(0,1),\qquad\theta\in\R,
\]
for which no common dominating measure---and hence no likelihood---exists, so
that no likelihood-based estimator of the atom location $\theta$ can even be
written down.  The weak first moment is nevertheless available in closed form:
with the Gaussian kernel $\varphi_\sigma(x)=e^{-x^2/(2\sigma^2)}$ of bandwidth $\sigma$, the $N(0,1)$ part contributes nothing
by symmetry and
\[
   {}^{(\varphi_\sigma)}m_1(\theta)
   =\langle T_\theta,x\,\varphi_\sigma\rangle
   =\tfrac12\,\theta\,e^{-\theta^2/2\sigma^2},
\]
which is strictly increasing on $(-\sigma,\sigma)$.  Given i.i.d.\
$X_1,\dots,X_n\sim P_\theta$, the empirical weak first moment
$\widehat m_1=n^{-1}\sum_{i=1}^n X_i\varphi_\sigma(X_i)$ obeys the strong law
of large numbers (the summand $x\varphi_\sigma(x)$ is bounded), and the
estimator $\widehat\theta$ solving
$\widehat m_1={}^{(\varphi_\sigma)}m_1(\widehat\theta)$ exists, is unique on
$(-\sigma,\sigma)$, and---by the same implicit-function argument---is
$\sqrt n$-consistent and asymptotically normal there.  Thus the atom location
is consistently estimable from a single weak moment in a model that admits
\emph{no likelihood at all}: the exact counterpart, on the
likelihood-failure side, of the Cauchy estimator on the moment-failure side.
A fuller treatment---the asymptotic variance, the bandwidth/identifiability
trade-off, and a GMM with several weak moments to widen the range
$|\theta|<\sigma$---is given in the companion paper~\cite{LabouriauA2}.

\section{Discussion and open problems}
\label{sec:discussion}

The framework developed in this paper represents a probability law by a
tempered distribution $T$ rather than by a density, and extracts information
from it through the kernel-regularised pairing $\langle T,\psi\varphi\rangle$. From this single
device we obtain a coherent extension of moment-based probability:
weak moments of all orders exist unconditionally; the weak characteristic
function is smooth and encodes these moments; weak cumulants retain their
algebraic structure under independence and affine transformations; a weak
moment problem admits a hierarchy of uniqueness results; and a weak central
limit theorem holds at the level of regularised transforms.

These results show that the absence of classical moments does not prevent a
meaningful probabilistic calculus. Instead, the kernel $\varphi$ provides a
controlled regularisation that restores finiteness and smoothness while
preserving the essential algebraic structure of moments and cumulants.

\medskip

Several questions remain open and define a natural research programme.

\emph{Uniqueness beyond the hierarchy and classical M-determinacy:}
For which positive Schwartz kernels is the family
$\{x^n\varphi:n\ge 0\}$ total in $\mathcal{S}(\R)$?
The Gaussian kernel provides a canonical case, via the completeness of the
Hermite functions. More generally, this problem is closely related to
weighted polynomial approximation on the real line and to the classical
theory of M-determinacy.  As discussed in
Section~\ref{subsec:classical-moment-problem}, the weak moment problem
transforms the classical Hamburger moment problem into a weighted one:
Carleman's condition and Krein's criterion~\cite{Stoyanov2000Krein} are applied to
the weighted measure $\varphi(x)\,dx$ rather than to the original
distribution.  The boundary between M-determinate and M-indeterminate
weighted measures---explored in depth by
Stoyanov~\cite{Stoyanov2013,StoyanovInverardiTagliani2023} and
collaborators---therefore governs the reach of the weak uniqueness
theorems.  For kernels of the form $\varphi(x)=e^{-Q(x)}$ with $Q$ even
and sufficiently convex, one enters the classical theory of Freud weights,
where density and approximation properties of polynomials in weighted
function spaces have been extensively
studied~\cite{Lubinsky2007}. In particular, for
$\varphi(x)=e^{-|x|^\alpha}$ with $\alpha>1$, results on rapid polynomial
approximation suggest that the Gaussian case $\alpha=2$ is part of a broader
phenomenon, although a full characterisation of totality in
$\mathcal{S}(\R)$ remains open. Complementary insight is provided by the
theory of de~Branges, which characterises non-density of polynomials in
weighted spaces in terms of associated entire functions~\cite{deBranges1968}.
Extending such criteria to the Schwartz topology
appears to be the natural route to a complete solution.

A deeper geometric explanation for the resolution of
M-indeterminacy is developed in a companion
paper~\cite{LabouriauTransversality}.  In that work, classical
statistical degeneracies---including moment indeterminacy, failure of
identifiability, and non-existence of moments---are interpreted as
failures of \emph{transversality} between the statistical model and
a probing family (here, $\{x^n\varphi\}$) in an appropriate function
space.  Multiplication by a suitable kernel~$\varphi$ acts as a
geometric perturbation that restores the transversality lost in the
classical (unweighted) setting.  Concretely, the passage from the
bare monomial family $\{x^n\}$ to the weighted family
$\{x^n\varphi\}$ moves the probing subspace into general position
relative to the model, so that the intersection---the set of
distributions sharing all weak moments---shrinks to a single point.
From this viewpoint, the uniqueness hierarchy of
Section~\ref{sec:moment_problem} quantifies how much decay the
kernel must have (exponential tails, Gevrey-type decay, Gaussian
decay) to guarantee that transversality is recovered.  This
geometric perspective also explains why some kernels fail: when
$\varphi$ decays too slowly, the perturbation is insufficient to
break the degeneracy, and M-indeterminacy persists (as in the
sub-exponential counterexample of
Remark~\ref{rem:carleman-automatic}).

\medskip

\emph{Identifiability in parametric families:}
While the weak moment sequence determines the underlying distribution under
conditions on the kernel, parametric identifiability requires that the map
$\theta\mapsto\{{}^{(\varphi)}m_n(\theta)\}$ be injective. This is a
model-dependent question that must be analysed case by case. In the context
of statistical inference, related local identifiability conditions arise
through rank assumptions on moment mappings, and are developed further in
the companion paper.

\medskip

\emph{Rates in the weak CLT:}
The weak central limit theorem of Section~\ref{sec:clt} is already
quantitative---a Berry--Esseen-type bound in transform space
(Theorem~\ref{thm:weak-BE}), refined to a rate that is \emph{uniform over the
kernel} (Theorem~\ref{thm:weak-BE-uniform}, instrument-invariance).  What
remains open is the transfer of such rates to genuine distributional
metrics---on the kernel-weighted or reconstructed distribution function,
rather than on weak characteristic functions---which would combine the
transform bounds with the reconstruction analysis of
Appendix~\ref{app:reconstruction}.

\medskip

\emph{Singularity as differentiated regularity.}
The structure-theorem reading of Remark~\ref{rem:structure}---singular
statistical objects as differentiated regularity, stabilised by the
kernel---is more than interpretation.  It situates the present framework within
a broader programme in which classical statistical degeneracies
(non-identifiability, singular information, moment indeterminacy) are read as
unstable configurations of differentiated regularity and resolved by kernel and
geometric regularisation~\cite{LabouriauTransversality}.  It also indicates the
reach of the construction to genuinely singular models---those with no density
and no classical inferential footing---where the kernel supplies the only
stable scalar quantities available.

\medskip

\emph{Finite-resolution observation.}
The framework also sits within the approximation principles of distribution
theory.  A Schwartz mollifier
$\rho_\varepsilon(x)=\varepsilon^{-d}\rho(x/\varepsilon)$ with $\int\rho=1$
satisfies $\rho_\varepsilon\to\delta_0$ in $\mathcal{S}'(\R^d)$ as
$\varepsilon\downarrow0$, so classical pointwise observation is the
infinite-resolution limit of smooth observational schemes.  The weak framework
reverses the usual idealisation: it takes finite-resolution observation through
the kernel as primitive, and recovers the classical, kernel-free quantities in
the complementary limit $\varphi\to1$ (Lemma~\ref{lem:phi-to-1}).  This reading
is especially apt in the heavy-tailed, singular, and non-likelihood settings of
this paper, where the infinite-resolution object may itself be unstable or
unavailable.

\medskip

\emph{Connections with other frameworks.}
The weak-expectation framework is related in spirit to several extensions of
classical probability, but is distinct from each of them. In sublinear
expectation theory \cite{Peng2008}, the expectation operator is nonlinear
and encodes model uncertainty, leading to objects such as the $G$-normal
distribution and nonlinear central limit theorems. In contrast, the present
framework retains linearity once the pair $(T,\varphi)$ is fixed, and the
extension arises from regularisation rather than nonlinearity. It is closer
in spirit to the theory of generalised random variables and processes in the
sense of Gel'fand and Vilenkin \cite{Gelfand1964}, where random objects are
defined as continuous linear functionals on spaces of test functions and
characterised by their characteristic functionals. The present approach may
be viewed as a scalar, regularised counterpart of this theory, focused on
moment, cumulant, and transform structures. Finally, there is a connection
to functional limit theorems for heavy-tailed processes \cite{Resnick2007},
where partial sums converge to stable or self-similar limits. The weak CLT
developed here operates at the level of regularised transforms for scalar
sums; extending it to process-level limits in distribution or Skorokhod
spaces is a natural direction for future work.

\medskip

\emph{Statistical inference.}
The motivating consequence presented in
Section~\ref{sec:motivating-consequence} is deliberately minimal. A full
development---including weak moment estimation, transform-based methods,
regularised reconstruction, optimal kernel selection, and robustness
analysis---is the subject of a companion paper (see \cite{LabouriauA2}). In particular, the weak
framework naturally leads to redescending $M$-estimators (or inference functions) with bounded
influence functions, without the need for ad hoc truncation (see  \cite{LabouriauPaperD}).

\medskip

These questions define a coherent programme growing out of the present
framework, linking distributional probability, weighted approximation
theory, and robust statistical inference.


\newpage

\appendix

\section{Leibniz estimate for the proof of Theorem~\ref{thm:dc-unique}}
\label{app:leibniz}

We verify in detail the bound~\eqref{eq:psi-gevrey} claimed in Part~1 of
the proof of Theorem~\ref{thm:dc-unique}.

\medskip
\noindent\textbf{Setup.}
Recall that $S = T_1 - T_2\in\mathcal{S}'(\R)$ and
$\psi(t) = \langle S, e^{itx}\varphi(x)\rangle$.  Writing
$f_n(x) = (ix)^n e^{itx}\varphi(x)$, we have
$\psi^{(n)}(t) = \langle S, f_n\rangle$.  Since
$S\in\mathcal{S}'(\R)$, there exist $N\in\N$ and $C'>0$ such that
\begin{equation}\label{eq:S-bound}
  |\langle S, f\rangle|
  \;\le\;
  C'\!\sum_{\substack{\alpha,\gamma\ge 0 \\ \alpha+\gamma\le N}}
  \sup_{x\in\R}\,|x^\alpha f^{(\gamma)}(x)|
  \qquad\text{for all }f\in\mathcal{S}(\R).
\end{equation}
It therefore suffices to bound each seminorm
$p_{\alpha,\gamma}(f_n) = \sup_x |x^\alpha f_n^{(\gamma)}(x)|$
with $\alpha+\gamma\le N$.

\medskip
\noindent\emph{Leibniz expansion:}
Fix a pair $(\alpha,\gamma)$ with $\alpha+\gamma\le N$.  The
$\gamma$-th $x$-derivative of $f_n(x)=(ix)^n e^{itx}\varphi(x)$ is
computed by the Leibniz rule applied to the three factors
$u(x)=(ix)^n$, $v(x)=e^{itx}$, $w(x)=\varphi(x)$:
\begin{equation}\label{eq:leibniz-triple}
  f_n^{(\gamma)}(x)
  \;=\;
  \sum_{j+k+\ell=\gamma}
  \frac{\gamma!}{j!\,k!\,\ell!}\;
  [(ix)^n]^{(j)}\;\cdot\;
  [e^{itx}]^{(k)}\;\cdot\;
  \varphi^{(\ell)}(x).
\end{equation}
The individual factors are:
\begin{itemize}
\item $[(ix)^n]^{(j)} = i^n \cdot \frac{n!}{(n-j)!}\,x^{n-j}$ for
$j\le n$ (and zero for $j>n$), so
$|[(ix)^n]^{(j)}| = \frac{n!}{(n-j)!}\,|x|^{n-j}$.
\item $[e^{itx}]^{(k)} = (it)^k e^{itx}$, so
$|[e^{itx}]^{(k)}| = |t|^k$.
\item $\varphi^{(\ell)}(x)$ is controlled by the decay
condition~\eqref{eq:gevrey-decay}.
\end{itemize}

\medskip
\noindent\emph{Applying the decay condition:}
For each term in~\eqref{eq:leibniz-triple}, the contribution to the
seminorm $p_{\alpha,\gamma}$ requires bounding
\[
  |x|^\alpha \cdot \frac{n!}{(n-j)!}\,|x|^{n-j}\cdot
  |t|^k\cdot|\varphi^{(\ell)}(x)|
  \;=\;
  \frac{n!}{(n-j)!}\,|t|^k\;
  |x|^{n-j+\alpha}\,|\varphi^{(\ell)}(x)|.
\]
Since $|x|^{n-j+\alpha}\le (1+|x|)^{n-j+\alpha}$, the decay
condition~\eqref{eq:gevrey-decay} (with $m=\ell$ and
$k$ replaced by $n-j+\alpha$) gives
\begin{equation}\label{eq:decay-applied}
  (1+|x|)^{n-j+\alpha}\,|\varphi^{(\ell)}(x)|
  \;\le\;
  C_\ell\,A^{n-j+\alpha}\,\bigl((n-j+\alpha)!\bigr)^{1/\beta}.
\end{equation}

\medskip
\noindent\emph{Factorial estimate:}
Since $j\le\gamma\le N$ and $\alpha\le N$, we have
$n-j+\alpha \le n+N$.  We need to control
$((n-j+\alpha)!)^{1/\beta}$ in terms of $(n!)^{1/\beta}$.

For $n\ge N$, the identity $(n+N)! = (n+N)(n+N-1)\cdots(n+1)\cdot n!$
gives
\begin{equation}\label{eq:factorial-ineq}
  (n+N)!
  \;\le\; (n+N)^N\cdot n!
  \;\le\; (2n)^N\cdot n!.
\end{equation}
Since $n-j+\alpha\le n+N$, monotonicity of the factorial yields
$(n-j+\alpha)!\le (n+N)!\le (2n)^N\cdot n!$.  Raising to the power
$1/\beta$:
\begin{equation}\label{eq:factorial-beta}
  \bigl((n-j+\alpha)!\bigr)^{1/\beta}
  \;\le\; (2n)^{N/\beta}\,(n!)^{1/\beta}
  \;\le\; 2^{N/\beta}(n+1)^{N/\beta}\,(n!)^{1/\beta},
\end{equation}
where the last inequality uses $n\le n+1$ for tidiness.

\medskip
\noindent\emph{The binomial prefactor:}
The falling factorial $n!/(n-j)!$ satisfies, for $0\le j\le N$,
\[
  \frac{n!}{(n-j)!}
  \;=\; n(n-1)\cdots(n-j+1)
  \;\le\; n^j
  \;\le\; (n+1)^N.
\]

\medskip
\noindent\emph{Assembling the bound:}
Combining Steps~2--4, each term in the Leibniz
expansion~\eqref{eq:leibniz-triple} contributes at most
\[
  \frac{\gamma!}{j!\,k!\,\ell!}\;
  (n+1)^N\;\cdot\;|t|^k\;\cdot\;
  C_\ell\,A^{n-j+\alpha}\;\cdot\;
  2^{N/\beta}(n+1)^{N/\beta}\,(n!)^{1/\beta}
\]
to $\sup_x|x^\alpha f_n^{(\gamma)}(x)|$.  For $|t|\le R$, we bound
$|t|^k\le R^k\le R^N$.  Since $A^{n-j+\alpha}\le A^{n+N}
= A^N\cdot A^n$ and the number of terms in the
sum~\eqref{eq:leibniz-triple} is $\binom{\gamma+2}{2}\le
\binom{N+2}{2}$, there exists a constant $E = E(N,\beta,R,A,\{C_\ell\})>0$
such that
\[
  p_{\alpha,\gamma}(f_n)
  \;\le\;
  E\,(n+1)^{N+N/\beta}\,A^n\,(n!)^{1/\beta}.
\]

\medskip
\noindent\emph{Concluding:}
Summing over the finitely many pairs $(\alpha,\gamma)$ with
$\alpha+\gamma\le N$ and applying~\eqref{eq:S-bound}:
\[
  |\psi^{(n)}(t)|
  \;\le\;
  C'\!\sum_{\alpha+\gamma\le N} p_{\alpha,\gamma}(f_n)
  \;\le\;
  \underbrace{C'\cdot\binom{N+2}{2}\cdot E}_{=:\,C''}\;
  (n+1)^{N+N/\beta}\,A^n\,(n!)^{1/\beta}.
\]
Since $N+N/\beta\le 2N$ (as $\beta\ge 1$), we may absorb the
exponent into $(n+1)^{2N}$ and set $B=A$ to obtain the
bound~\eqref{eq:psi-gevrey}:
\[
  |\psi^{(n)}(t)|
  \;\le\;
  C''\,(n+1)^{N'}\,B^n\,(n!)^{1/\beta}
  \qquad\text{for all }|t|\le R,\; n\ge 0,
\]
with $N' = 2N$ and constants $B,C''>0$ depending on $R$, $N$,
$\beta$, $A$, and the seminorm constants
$\{C_\ell\}_{\ell=0}^N$.  This completes the verification
of~\eqref{eq:psi-gevrey}.

\section{Supplementary proofs}
\label{app:proofs}

This appendix collects the proofs of several results stated in the main text.  These proofs may be omitted from the journal version without affecting the logical structure of the paper.

\medskip
\noindent\textbf{Proof of Corollary~\ref{cor:cumulant-additivity}}
(Additivity of weak cumulants).

\begin{proof}
By Theorem~\ref{thm:cumulant-additivity}, the weak cumulant generating
functions satisfy
${}^{(\varphi)}K_{X+Y}(t) = {}^{(\varphi_1)}K_X(t) + {}^{(\varphi_2)}K_Y(t)$
in a neighbourhood of $t=0$.  Since ${}^{(\varphi_1)}\phi_X(0)\neq 0$ and
${}^{(\varphi_2)}\phi_Y(0)\neq 0$, all three CGFs are smooth near the
origin.  Differentiating $n$ times with respect to $t$ and evaluating
at $t=0$ gives
\[
\frac{d^n}{dt^n}{}^{(\varphi)}K_{X+Y}(t)\Big|_{t=0}
=
\frac{d^n}{dt^n}{}^{(\varphi_1)}K_X(t)\Big|_{t=0}
+
\frac{d^n}{dt^n}{}^{(\varphi_2)}K_Y(t)\Big|_{t=0},
\]
which is ${}^{(\varphi)}\kappa_n(X+Y) = {}^{(\varphi_1)}\kappa_n(X)
+ {}^{(\varphi_2)}\kappa_n(Y)$.
\end{proof}

\medskip
\noindent\textbf{Proof of Proposition~\ref{prop:translation}}
(Translation).

\begin{proof}
Recall that the translated pair $(T_a,\varphi_a)$ is defined by
$\langle T_a,\psi\rangle = \langle T,\psi(\cdot+a)\rangle$ and
$\varphi_a(x) = \varphi(x-a)$.  Hence the weak characteristic function
of the translated pair is
\begin{align*}
{}^{(\varphi_a)}\phi_a(t)
&= \langle T_a,\,e^{itx}\varphi_a(x)\rangle
= \langle T,\,e^{it(x+a)}\varphi_a(x+a)\rangle \\
&= \langle T,\,e^{it(x+a)}\varphi((x+a)-a)\rangle
= \langle T,\,e^{it(x+a)}\varphi(x)\rangle \\
&= e^{ita}\,\langle T,\,e^{itx}\varphi(x)\rangle
= e^{ita}\,{}^{(\varphi)}\phi(t).
\end{align*}
Taking logarithms, $K_a(t) = ita + {}^{(\varphi)}K(t)$.  Differentiating:
$K_a'(0) = ia + {}^{(\varphi)}K'(0)$, giving
${}^{(\varphi_a)}\kappa_1 = a + {}^{(\varphi)}\kappa_1$.  For $n\geq 2$,
the term $ita$ contributes zero to the $n$-th derivative, so
${}^{(\varphi_a)}\kappa_n = {}^{(\varphi)}\kappa_n$.
\end{proof}

\medskip
\noindent\textbf{Proof of Proposition~\ref{prop:scaling}}
(Scaling).

\begin{proof}
The scaled pair $(T_b,\varphi_b)$ is defined operationally by
$\E_{T_b,\varphi_b}[\psi] = \E_{T,\varphi}[\psi(b\,\cdot)]$.  Setting
$\psi(x)=e^{itx}$:
\[
{}^{(\varphi_b)}\phi_b(t)
= \E_{T_b,\varphi_b}[e^{itx}]
= \E_{T,\varphi}[e^{itbx}]
= {}^{(\varphi)}\phi(bt).
\]
Taking logarithms, $K_b(t) = {}^{(\varphi)}K(bt)$.  By the chain rule,
\[
\frac{d^n}{dt^n}K_b(t)\Big|_{t=0}
= b^n\,\frac{d^n}{ds^n}{}^{(\varphi)}K(s)\Big|_{s=0},
\]
so ${}^{(\varphi_b)}\kappa_n = b^n\,{}^{(\varphi)}\kappa_n$.
\end{proof}

\medskip
\noindent\textbf{Proof of Proposition~\ref{prop:wcf-smooth}}
(Smoothness of the weak characteristic function).

\begin{proof}
We must show that $t\mapsto {}^{(\varphi)}\phi(t) = \langle T,
e^{itx}\varphi(x)\rangle$ is infinitely differentiable in~$t$ and
that derivatives pass through the distributional pairing.

\emph{Step~1.  The map $t\mapsto g_t$ is smooth from $\R$ into
$\mathcal{S}(\R)$.}
Define $g_t(x) = e^{itx}\varphi(x)$.  For each fixed~$t$,
$g_t\in\mathcal{S}(\R)$ because $|e^{itx}|=1$ and $\varphi\in
\mathcal{S}(\R)$, and products with bounded smooth functions of
polynomial growth preserve the Schwartz class.  The formal $t$-derivative
is $\partial_t g_t(x) = ix\,e^{itx}\varphi(x)$, and more generally
$\partial_t^n g_t(x) = (ix)^n e^{itx}\varphi(x)$.  Since
$x^n\varphi(x)\in\mathcal{S}(\R)$ (the Schwartz space is closed
under polynomial multiplication), each $\partial_t^n g_t$ belongs
to $\mathcal{S}(\R)$.

\emph{Step~2.  Convergence of difference quotients in
$\mathcal{S}(\R)$.}
We verify that $h^{-1}(g_{t+h}-g_t)\to \partial_t g_t$ in the
Schwartz topology as $h\to 0$.  The difference is
\[
h^{-1}(g_{t+h}(x)-g_t(x))
= e^{itx}\varphi(x)\cdot\frac{e^{ihx}-1}{h}.
\]
By the mean value theorem,
$(e^{ihx}-1)/h = ix\,e^{i\theta hx}$ for some
$\theta\in(0,1)$ depending on $h$ and~$x$.  Thus
\[
h^{-1}(g_{t+h}(x)-g_t(x)) - ix\,e^{itx}\varphi(x)
= ix\,e^{itx}\varphi(x)(e^{i\theta hx}-1).
\]
For any Schwartz seminorm
$p_{\alpha,\beta}(f)=\sup_x|x^\alpha f^{(\beta)}(x)|$, the
Leibniz rule applied to $x$-derivatives of $ix\,e^{itx}\varphi(x)
(e^{i\theta hx}-1)$ produces finitely many terms, each containing
$|e^{i\theta hx}-1|\le |\theta hx|\le |hx|$ times a Schwartz
function of~$x$.  Since $|hx|\cdot|x^m\varphi^{(k)}(x)|\le
|h|\sup_x|x|^{m+1}|\varphi^{(k)}(x)|$ and $\varphi\in\mathcal{S}$,
every seminorm tends to zero as $h\to 0$.  The same argument applies
to higher-order difference quotients by induction.

\emph{Step~3.  Conclusion.}
Since $T\in\mathcal{S}'(\R)$ is a continuous linear functional on
$\mathcal{S}(\R)$, the composition $t\mapsto\langle T,g_t\rangle$
is smooth and
\[
\frac{d^n}{dt^n}\langle T,g_t\rangle
= \langle T,\partial_t^n g_t\rangle
= \langle T,(ix)^n e^{itx}\varphi(x)\rangle
= \E_{T,\varphi}[(ix)^n e^{itx}]. \qedhere
\]
\end{proof}

\medskip
\noindent\textbf{Proof of Corollary~\ref{cor:moments-from-cf}}
(Moments from derivatives at the origin).

\begin{proof}
Set $t=0$ in Proposition~\ref{prop:wcf-smooth}:
\[
({}^{(\varphi)}\phi)^{(n)}(0)
= \langle T,(ix)^n e^{i\cdot 0\cdot x}\varphi(x)\rangle
= i^n\langle T,x^n\varphi(x)\rangle
= i^n\,{}^{(\varphi)}m_n. \qedhere
\]
\end{proof}

\medskip
\noindent\textbf{Proof of Theorem~\ref{thm:weak-clt}}
(Weak Central Limit Theorem).

\begin{proof}
Let ${}^{(\varphi)}\phi$ denote the common weak characteristic function
of the i.i.d.\ generalised random variables $X_1,\ldots,X_n$ with
first two weak cumulants ${}^{(\varphi)}\kappa_1$ and
${}^{(\varphi)}\kappa_2 > 0$.

\emph{Step~1.  Weak CF of the normalised sum.}
By the factorisation property (Proposition~\ref{prop:cf-factorise})
applied to i.i.d.\ copies,
${}^{(\varphi)}\phi_{S_n}(t) = ({}^{(\varphi)}\phi(t))^n$.  The
normalisation $Z_n = (S_n - n{}^{(\varphi)}\kappa_1)/\sqrt{n{}^{(\varphi)}\kappa_2}$
corresponds, via translation and scaling
(Propositions~\ref{prop:translation} and~\ref{prop:scaling}), to the
logarithmic identity
\[
\log{}^{(\varphi)}\phi_{Z_n}(t)
= -it\frac{n{}^{(\varphi)}\kappa_1}{\sqrt{n{}^{(\varphi)}\kappa_2}}
+ n\log{}^{(\varphi)}\phi\!\left(\frac{t}{\sqrt{n{}^{(\varphi)}\kappa_2}}\right).
\]

\emph{Step~2.  Second-order expansion.}
Since ${}^{(\varphi)}\phi$ is smooth (Proposition~\ref{prop:wcf-smooth})
and ${}^{(\varphi)}\phi(0)=\langle T,\varphi\rangle\neq 0$, the function
$\log{}^{(\varphi)}\phi$ is smooth near the origin with
\[
\log{}^{(\varphi)}\phi(u)
= i{}^{(\varphi)}\kappa_1 u - \frac{{}^{(\varphi)}\kappa_2}{2}u^2 + r(u),
\qquad r(u)=o(u^2)\text{ as }u\to 0.
\]

\emph{Step~3.  Substitution and limit.}
Write $\mu={}^{(\varphi)}\kappa_1$ and $\sigma^2={}^{(\varphi)}\kappa_2$
for brevity and set $u=t/\sqrt{n\sigma^2}$:
\begin{align*}
\log{}^{(\varphi)}\phi_{Z_n}(t)
&= -\frac{it\sqrt{n}\,\mu}{\sigma}
+ n\!\left[
i\mu\frac{t}{\sqrt{n\sigma^2}}
-\frac{\sigma^2}{2}\frac{t^2}{n\sigma^2}
+ r\!\left(\frac{t}{\sqrt{n\sigma^2}}\right)
\right] \\
&= -\frac{it\sqrt{n}\,\mu}{\sigma}
+ \frac{it\sqrt{n}\,\mu}{\sigma}
- \frac{t^2}{2}
+ n\,r\!\left(\frac{t}{\sqrt{n\sigma^2}}\right) \\
&= -\frac{t^2}{2}
+ n\,r\!\left(\frac{t}{\sqrt{n\sigma^2}}\right).
\end{align*}
For the remainder: since $r(u)=o(u^2)$, for every $\varepsilon>0$
there exists $\delta>0$ such that $|r(u)|\le\varepsilon|u|^2$
whenever $|u|<\delta$.  For $n$ large enough,
$|t|/\sqrt{n\sigma^2}<\delta$, hence
\[
\left|n\,r\!\left(\frac{t}{\sqrt{n\sigma^2}}\right)\right|
\le n\,\varepsilon\,\frac{t^2}{n\sigma^2}
= \frac{\varepsilon\,t^2}{\sigma^2}
\;\xrightarrow{\varepsilon\to 0}\; 0.
\]
Therefore $\log{}^{(\varphi)}\phi_{Z_n}(t)\to -t^2/2$, and consequently
${}^{(\varphi)}\phi_{Z_n}(t)\to e^{-t^2/2}$ for every $t\in\R$.
\end{proof}

\section{Reconstruction from weak data}
\label{app:reconstruction}

\medskip
\noindent
We show that, in the classical density setting, the weak
representation admits a stable reconstruction via Tikhonov
regularisation. In the terms of this paper, the kernel $\varphi$ is the
\emph{observation instrument}: it yields the kernel-weighted data
$g=\varphi f$, and recovering the law amounts to \emph{inverting the
instrument}. This provides an explicit inverse procedure complementing
the uniqueness results of Section~\ref{sec:moment_problem}.

\medskip

Let $\varphi\in\mathcal{S}(\R)$ with $\varphi(x)>0$ for all $x$, and let
$f\in L^2(\R)$. In the classical case $T=T_f$, the weak expectation
determines the product
\[
g(x) := \varphi(x) f(x),
\]
either directly (through admissible test functions) or indirectly via the
weak characteristic function
\[
{}^{(\varphi)}\phi(t)
=
\int_{\R} e^{itx}\varphi(x)f(x)\,dx
=
\widehat{g}(t).
\]
Thus, reconstruction of $f$ from weak data reduces to the inversion of the
multiplication operator
\[
M_\varphi : L^2(\R) \to L^2(\R),
\qquad
(M_\varphi f)(x) = \varphi(x) f(x).
\]

Since $\varphi(x)\to 0$ as $|x|\to\infty$, the inverse problem
$M_\varphi f = g$ is ill-posed: small perturbations in $g$ can lead to large
errors in $f$ where $\varphi$ is small. A natural remedy is Tikhonov
regularisation.

\begin{theorem}[Tikhonov reconstruction from weak data]
\label{thm:tikhonov-reconstruction}
Let $\varphi\in\mathcal{S}(\R)$ satisfy $\varphi(x)>0$ for all $x$, and
let $M_\varphi$ be as above. Given data $g\in L^2(\R)$, define for
$\lambda>0$
\[
R_\lambda g
:=
(M_\varphi^* M_\varphi + \lambda I)^{-1} M_\varphi^* g.
\]
Then:

\begin{enumerate}
\item $R_\lambda g$ is the unique minimiser of the functional
\[
J_\lambda(h)
=
\|M_\varphi h - g\|_{L^2}^2
+
\lambda \|h\|_{L^2}^2.
\]

\item If $g = \varphi f$ with $f\in L^2(\R)$, then
\[
R_\lambda g \to f
\quad\text{in }L^2(\R)
\quad\text{as }\lambda\downarrow 0.
\]

\item If noisy data $g_\delta$ satisfy
\[
\|g_\delta - g\|_{L^2} \le \delta,
\]
then
\[
\|R_\lambda g_\delta - f\|_{L^2}
\le
\frac{\delta}{2\sqrt{\lambda}}
+
\left\|
\frac{\lambda}{\varphi^2 + \lambda} f
\right\|_{L^2}.
\]
In particular, if $\lambda=\lambda(\delta)\downarrow 0$ and
$\delta/\sqrt{\lambda(\delta)}\to 0$, then
\[
R_{\lambda(\delta)} g_\delta \to f
\quad\text{in }L^2(\R).
\]
\end{enumerate}
\end{theorem}

\begin{proof}
Since $M_\varphi$ is a bounded self-adjoint multiplication operator on
$L^2(\R)$, we have
\[
M_\varphi^* = M_\varphi,
\qquad
M_\varphi^* M_\varphi = M_{\varphi^2}.
\]
Hence
\[
M_\varphi^* M_\varphi + \lambda I
=
M_{\varphi^2 + \lambda},
\]
which is invertible for all $\lambda>0$. Its inverse is again a
multiplication operator:
\[
(M_{\varphi^2 + \lambda})^{-1}
=
M_{(\varphi^2 + \lambda)^{-1}}.
\]
Therefore,
\[
R_\lambda g
=
M_{(\varphi^2 + \lambda)^{-1}} M_\varphi g
=
\frac{\varphi}{\varphi^2 + \lambda}\, g.
\]

\medskip
\noindent
(1) The minimiser of $J_\lambda$ is characterised by the normal equation
\[
(M_\varphi^* M_\varphi + \lambda I)h = M_\varphi^* g,
\]
which yields the stated expression for $R_\lambda g$. Strict convexity of
$J_\lambda$ ensures uniqueness.

\medskip
\noindent
(2) If $g=\varphi f$, then
\[
R_\lambda g
=
\frac{\varphi^2}{\varphi^2 + \lambda} f,
\]
and hence
\[
R_\lambda g - f
=
-\frac{\lambda}{\varphi^2 + \lambda} f.
\]
Since $0 \le \frac{\lambda}{\varphi^2 + \lambda} \le 1$ and this factor
converges pointwise to $0$, dominated convergence yields
\[
\|R_\lambda g - f\|_{L^2} \to 0.
\]

\medskip
\noindent
(3) For noisy data $g_\delta$, write
\[
R_\lambda g_\delta - f
=
\frac{\varphi}{\varphi^2 + \lambda}(g_\delta - g)
-
\frac{\lambda}{\varphi^2 + \lambda} f.
\]
Taking norms,
\[
\|R_\lambda g_\delta - f\|_{L^2}
\le
\left\|
\frac{\varphi}{\varphi^2 + \lambda}
\right\|_\infty
\|g_\delta - g\|_{L^2}
+
\left\|
\frac{\lambda}{\varphi^2 + \lambda} f
\right\|_{L^2}.
\]
The function $x \mapsto \frac{x}{x^2 + \lambda}$ attains its maximum at
$x=\sqrt{\lambda}$ with value $1/(2\sqrt{\lambda})$, hence
\[
\left\|
\frac{\varphi}{\varphi^2 + \lambda}
\right\|_\infty
\le
\frac{1}{2\sqrt{\lambda}}.
\]
This yields the stated bound and the convergence condition.
\end{proof}

\begin{remark}
The reconstruction operator $R_\lambda$ regularises the ill-posed inversion
of $M_\varphi$ by suppressing regions where $\varphi$ is small. The
regularisation parameter $\lambda$ controls the trade-off between bias and
stability. This illustrates that the weak representation, while losing
pointwise information through multiplication by $\varphi$, retains enough
structure to allow stable recovery of the underlying density under suitable
regularisation.
\end{remark}

\begin{remark}[On the $L^2$ assumption]
Theorem~\ref{thm:tikhonov-reconstruction} assumes $f\in L^2(\R)$, which is
slightly stronger than the $L^1$ integrability assumed for densities
elsewhere in the paper.  This is a mild condition: since $\varphi$ is
bounded, every $L^2(\R)$ density qualifies, and virtually all standard
parametric families (Gaussian, Student~$t$, stable with index $\alpha > 1/2$,
generalised hyperbolic, etc.)\ have $L^2$ densities.  The $L^2$ setting is
natural for the Tikhonov framework because the normal equation
$(M_\varphi^*M_\varphi + \lambda I)h = M_\varphi^*g$ is formulated in
Hilbert space.
\end{remark}


\begin{thebibliography}{99}

\bibitem{Berg1988}
C.~Berg.
\newblock The cube of a normal distribution is indeterminate.
\newblock {\em Annals of Probability}, 16(2):910--913, 1988.

\bibitem{Billingsley1995}
P.~Billingsley (1995).
\textit{Probability and Measure}, 3rd ed.
Wiley.

\bibitem{deBranges1968}
L. de Branges (1968).
\textit{Hilbert Spaces of Entire Functions}.
Prentice-Hall.

\bibitem{Feller1971}
W.~Feller (1971).
\textit{An Introduction to Probability Theory and Its Applications}, Vol.~II, 2nd ed.
Wiley.

\bibitem{Gelfand1964}
I.~M.~Gel'fand and N.~Ya.~Vilenkin (1964).
\textit{Generalized Functions, Vol.~4: Applications of Harmonic Analysis}.
Academic Press.

\bibitem{Hormander1990}
L.~H\"ormander (1990).
\textit{The Analysis of Linear Partial Differential Operators~I},
2nd ed. Springer, Berlin.

\bibitem{LabouriauMixture}
R.~Labouriau (2023).
On the bias of the score function of finite mixture models.
\textit{Communications in Statistics -- Theory and Methods}, 52(13), 4461--4467.
DOI: 10.1080/03610926.2021.1995429.

\bibitem{LabouriauDeMoivre}
R.~Labouriau (2026).
From Coefficients to Distributions:
De~Moivre and the Operational View of Probability.
arXiv:2605.25227 [math.HO]

\bibitem{LabouriauPaperD}
R.~Labouriau (2026).
Inference  Functionals and Observation Operators for Distributional
Statistical Models.
arXiv:2605.19189 [math.ST].

\bibitem{LabouriauTransversality}
R.~Labouriau (2026).
Transversality and Geometric Regularisation in Distributional
Statistical Models.
arXiv:2605.04536 [math.ST]

\bibitem{LabouriauA2}
R.~Labouriau (2026).
Weak Moment Methods for Statistical Inference:
with an Application to Robust Estimation.
arXiv:2604.23619 [stat.ME]

\bibitem{LinKopanovStoyanov2020}
G.~D.~Lin, P.~Kopanov, and J.~Stoyanov.
\newblock New checkable conditions for moment determinacy of probability distributions.
\newblock {\em Theory of Probability \& Its Applications}, 65(4):634--648, 2020.

\bibitem{Lubinsky2007}
D. S. Lubinsky (2007).
A survey of weighted polynomial approximation with exponential weights.
\textit{Surveys in Approximation Theory}, 3, 1--105.

\bibitem{Lukacs1970}
E.~Lukacs (1970).
\textit{Characteristic Functions}, 2nd ed.
Griffin.

\bibitem{Peng2008}
S. Peng (2008).
A new central limit theorem under sublinear expectations.
\textit{arXiv preprint} arXiv:0803.2656.

\bibitem{reed-simon1980}
M.~Reed and B.~Simon (1980).
\textit{Methods of Modern Mathematical Physics, Vol.~I:
Functional Analysis}, revised and enlarged ed.
Academic Press.

\bibitem{Resnick2007}
S. I. Resnick (2007).
\textit{Heavy-Tail Phenomena: Probabilistic and Statistical Modeling}.
Springer.

\bibitem{Samorodnitsky1994}
G.~Samorodnitsky and M.~S.~Taqqu (1994).
\textit{Stable Non-Gaussian Random Processes}.
Chapman \& Hall.

\bibitem{Schwartz1950}
L.~Schwartz (1950--1951).
\textit{Th\'eorie des distributions}.
Hermann.


\bibitem{Shohat1943}
J.~A.~Shohat and J.~D.~Tamarkin (1943).
\textit{The Problem of Moments}.
American Mathematical Society.

\bibitem{Stoyanov2000Krein}
J.~Stoyanov.
\newblock Krein condition in probabilistic moment problems.
\newblock {\em Bernoulli}, 6(5):939--949, 2000.

\bibitem{Stoyanov2013}
J.~Stoyanov.
\newblock {\em Counterexamples in Probability}.
\newblock Dover Publications, Mineola, NY, third edition, 2013.

\bibitem{StoyanovInverardiTagliani2023}
J.~Stoyanov, P.~L.~Inverardi, and A.~Tagliani.
\newblock The problem of moments: a bunch of classical results with some novelties.
\newblock {\em Symmetry}, 15(9):1743, 2023.

\bibitem{Strichartz2003}
R.~S.~Strichartz (2003).
\textit{A Guide to Distribution Theory and Fourier Transforms}.
World Scientific.

\bibitem{van2000asymptotic}
A.~W.~van der Vaart (1998).
\textit{Asymptotic Statistics}.
Cambridge University Press.

\bibitem{Watanabe2009}
S.~Watanabe (2009).
\textit{Algebraic Geometry and Statistical Learning Theory}.
Cambridge University Press.

\end{thebibliography}
\end{document}